\documentclass[a4paper,11pt]{amsart}

\usepackage{amssymb}
\usepackage{tikz}

\usepackage{bbm}

\usepackage[hmargin=3.5cm,foot=0.7cm]{geometry}
\usepackage[utf8x]{inputenc}

\usepackage[english]{babel}

\usepackage{tikz}
\usetikzlibrary{matrix,arrows}
\usepackage{enumitem}

\usepackage[initials]{amsrefs}

\newtheorem{theorem}{Theorem}[section]
\newtheorem*{theorem*}{Theorem}

\theoremstyle{plain}
\newtheorem{corollary}[theorem]{Corollary}
\newtheorem{lemma}[theorem]{Lemma}
\newtheorem{proposition}[theorem]{Proposition}
\newtheorem{conjecture}[theorem]{Conjecture}
\newcounter{mt}

\newtheorem{MainTheorem}[mt]{Theorem}

\newtheorem*{lemma*}{Lemma}
\newtheorem*{question*}{Question}

\theoremstyle{definition}
\newtheorem{definition}[theorem]{Definition}
\newtheorem{example}[theorem]{Example}
\newtheorem{remark}[theorem]{Remark}

\newcommand{\transv}{\mathrel{\text{\tpitchfork}}}
\makeatletter
\newcommand{\tpitchfork}{%
  \vbox{
    \baselineskip\z@skip
    \lineskip-.52ex
    \lineskiplimit\maxdimen
    \m@th
    \ialign{##\crcr\hidewidth\smash{$-$}\hidewidth\crcr$\pitchfork$\crcr}
  }%
}
\makeatother

\newcommand{\PP}{\mathbb{P}}
\newcommand{\RR}{\mathbb{R}}

\newcommand{\calK}{\mathcal{K}}

\newcommand{\calP}{\mathcal{P}}

\newcommand{\calV}{\mathcal{V}}

\newcommand \lcur{[\![}
\newcommand \rcur{]\!]}

\newcommand{\dt}{\left.\frac{d}{dt}\right|_{t=0}}

\DeclareMathOperator{\Val}{Val}
\DeclareMathOperator{\SO}{SO}

\DeclareMathOperator{\vol}{vol}

\DeclareMathOperator{\supp}{supp}

\DeclareMathOperator{\Grass}{Gr}
\DeclareMathOperator{\AGrass}{AGr}
\DeclareMathOperator{\Sym}{Sym}
\DeclareMathOperator{\sign}{sign}

\DeclareMathOperator{\Image}{Image}

\DeclareMathOperator{\Klain}{Kl}

\DeclareMathOperator{\Ker}{Ker}
\DeclareMathOperator{\Dens}{Dens}

\DeclareMathOperator{\nc}{nc}

\DeclareMathOperator{\codim}{codim}
\DeclareMathOperator{\rank}{rank}

\newcommand{\pder}[2][]{\frac{\partial#1}{\partial#2}}

\newtheorem{note}{Remark} 
\def\note#1{\ifvmode\leavevmode\fi\vadjust{\vbox to0pt{\vss
			\hbox to 0pt{\hskip\hsize\hskip1em
				\vbox{\hsize3.5cm\small\raggedright\pretolerance10000
					\noindent #1\hfill}\hss}\vbox to8pt{\vfil}\vss}}}

\title[]{The Weyl principle on the Finsler frontier} 
\author{Dmitry Faifman}
\author{Thomas Wannerer}
\email{dmitry.faifman@umontreal.ca}
\email{thomas.wannerer@uni-jena.de}
\address{Fakult\"at f\"ur Mathematik und Informatik, Friedrich-Schiller-Universit\"at Jena, 07743 Jena, Germany}
\address{Centre de recherches math\'ematiques, Universit\'e de Montr\'eal, Pavillon Andr\'e-Aisenstadt,
	2920, Chemin de la tour, bur. 5357 Montréal (Québec) H3T 1J4 Canada}

\thanks{DF partially supported by NSERC Discovery Grant.}
\thanks{TW supported by DFG grant WA 3510/1-1.}
\date{\today}

\begin{document}

	\begin{abstract}
			Any Riemannian manifold has a canonical collection of valuations (finitely additive measures) attached to it, known as the intrinsic volumes or Lipschitz-Killing valuations. They date back to the remarkable discovery of H. Weyl that the coefficients of the tube volume polynomial are intrinsic invariants of the metric. As a consequence, the intrinsic volumes behave naturally under isometric immersions. This phenomenon,  subsequently observed in a number of different geometric settings, is commonly referred to as the Weyl principle.
			In general normed spaces, the Holmes-Thompson intrinsic volumes naturally extend the Euclidean intrinsic volumes.
			The purpose of this note is to investigate the applicability  of the Weyl principle to Finsler manifolds. We show that while in general the Weyl principle fails, a weak form of the principle unexpectedly persists in certain settings.
\end{abstract}
\maketitle
\section{Introduction}
The relation between the intrinsic and extrinsic geometric properties of submanifolds of a finite-dimensional normed space is not well understood, see, e.g., \cites{burago-ivanov,ivanov, alvarez_problems}.  It is  not even clear which intrinsic and extrinsic quantities can be related to each other: D.~Burago and S.~Ivanov \cite{burago-ivanov} proved the striking  result that
 while geodesics on Finslerian saddle surfaces exhibit the same global behaviour as in the Riemannian case, there seem to be no special local properties of saddle surfaces that imply these global results, as every abstract Finsler surface can be locally embedded as a saddle surface in some $4$-dimensional normed space.

In convex geometry, the central global invariants of a convex body $K$ in Euclidean space $\RR^n$ are its \emph{intrinsic volumes}. They arise as the (suitably normalized) coefficients in Steiner's formula for the volume of an $\varepsilon$-neighborhood of $K$. The intrinsic volumes of $K\subset \RR^m\subset \RR^n$  do not depend on the ambient space, i.e. are the same whether computed in $\RR^m$ or in $\RR^n$. This fact was generalized to compact smooth submanifolds $M$ of $\RR^n$  by H.~Weyl \cite{weyl},  who expressed the intrinsic volumes in terms of polynomial invariants of the Riemannian curvature tensor of $M$. As a function on $\calK(\RR^n)$, the space of convex bodies in $\RR^n$, the $k$\textsuperscript{th} intrinsic volume $\mu_k$ possesses a strikingly simple characterization: $\mu_k$ is the unique even, $k$-homogeneous, translation-invariant,  continuous valuation, i.e., satisfies 
$$\mu_k(K\cap L)= \mu_k(K)+ \mu_k(L)-\mu_k(K\cap L)$$
whenever $K\cup L$ is convex, with the property that 
$$\mu_k|_E = \vol_E$$
for every $k$-dimensional linear subspace $E\subset \RR^n$, where $\vol_E$ denotes the usual Lebesgue measure on $E$.
This follows at once from the Klain embedding theorem \cites{klain_short,alesker_mcmullen} 

The Holmes-Thompson intrinsic volumes naturally extend the notion of the Euclidean intrinsic volumes to general smooth Minkowski spaces, i.e. finite-dimensional normed spaces with a smooth unit ball $B_F$ of positive gaussian curvature. 
There are several natural ways to normalize the Lebesgue measure on a normed space $V$, see e.g. \cite{alvarez-thompson}, or the book by A.C.~Thompson \cite{Thompson:MG}. 
From the perspective of integral geometry, as well as symplectic geometry, the natural definition is the \emph{Holmes-Thompson volume}, given by setting \[\vol^{HT}(B_F)=\frac{1}{\omega_n}\vol_{2n}(B_F\times B_F^\circ )\] where $B_F^\circ\subset V^*$ is the dual convex body,  $\omega_n$ the volume of the Euclidean unit ball, and $\vol_{2n}$ the Liouville volume for the standard symplectic form on $V\times V^*$. The  normalizing factor  $1/\omega_n$ is chosen so that $\vol^{HT}$ extends the standard Euclidean definition of volume. If $E\subset V$ is a linear subspace, it inherits a norm and therefore the corresponding Holmes-Thompson Lebesgue measure $\vol^{HT}_E$. More generally, given a $k$-dimensional submanifold $X\subset V$, we get the Holmes-Thompson volume measure $\vol^{HT}_X$ on $X$, since an absolutely continuous measure on a manifold is canonically identified with a continuous choice of a Lebesgue measure on its tangent spaces.
It was shown by J.C.~\'Alvarez Paiva and E.~Fernandes \cites{alvarez-fernandes} and A.~Bernig \cite{bernig_HT} that the $k$-dimensional Holmes-Thompson volume can be extended to a valuation on $\calK(V)$. Namely,  one can define the $k$\textsuperscript{th} \emph{Holmes-Thompson intrinsic volume} as the unique even, $k$-homogeneous, translation-invariant, continuous valuation satisfying 
$$\mu_k^F|_E = \vol^{HT}_E$$
for every $k$-dimensional linear subspace $E\subset V$.
The uniqueness of the extension of $\vol^{HT}_k$ follows as in the Euclidean case from the  Klain embedding theorem. As a consequence of this uniqueness, the Holmes-Thompson intrinsic volumes are intrinsic in the following sense: If $K\subset U\subset V$ is a convex body contained in a linear subspace $U\subset V$ with the induced norm, then the value of  $k$th Holmes-Thompson intrinsic volume  on $K$ is the same whether computed in $(V,F)$ or $(U,F|_U)$.  We will refer to this property as the \emph{linear Weyl principle}. As their Euclidean counterparts, the Holmes-Thompson intrinsic volumes may be evaluated on much more general objects than convex bodies, in particular on compact smooth submanifolds with boundary (see Section~\ref{sec:background}).
 
A guiding idea in valuation theory, usually referred to as the \emph{Weyl principle}, has recently emerged: when restricted to a subspace $X$, valuations  on an
ambient space $M$ may often be reconstructed from the basic geometry of $X$
induced by the immersion into $M$. 
The prototypical application is Weyl's theorem described above, but more recently several other instances have  surfaced:  immersions of contact and dual Heisenberg manifolds and restriction of their canonical valuations \cite{faifman_contact};  restriction of the invariant valuations on complex space forms  to totally real submanifolds \cite{bfs}*{Lemma 4.4}; and  immersions of pseudo-Riemannian manifolds and the restriction of their canonical valuations \cite{BFaS}.

Approaching the question of intrinsic and extrinsic geometric properties of submanifolds of a finite-dimensional normed space from an integral geometric perspective, J.H.G.~Fu put forward the conjecture \cite{alesker_fu_barcelona}*{p.~107}  that the Weyl principle also applies to the  Holmes-Thompson intrinsic volumes. More precisely, let $V$ be an $n$-dimensional real vector space, and let $F\colon V\to \RR$ be a smooth norm on $V$. The pullback of a valuation $\mu$ on $V$ under a smooth embedding $e\colon M\to V$ is defined by 
$$ e^* \mu(A)= \mu(e(A)),$$
where $A\subset M$ is a compact submanifold with corners (see Section~\ref{sec:background}).  This definition extends naturally to smooth immersions.
The conjecture of J.H.G.~Fu can now be stated as follows.
\begin{conjecture}[{\cite{alesker_fu_barcelona}*{p.~107}}]\label{conj:fu} Let $e_j\colon M\to V_j$, $j=1,2$ be smooth immersions of a smooth manifold $M$ into the normed spaces $(V_j,F_j)$. If $e_1^*F_1 = e_2^*F_2$, then
	$$e_1^* \mu_k^{F_1} = e_2^* \mu_k^{F_2}. $$
\end{conjecture}

When $\dim M\leq 2$, the conjecture is easily seen to hold (Theorem~\ref{thm:dim2}). The first unknown case is $\dim M=3$, $\dim V=4$.

D.~Burago and S.~Ivanov proved \cite{burago_ivanov_nash} a Burstin-Janet-Cartan theorem  in the Finsler setting, namely that every Finsler manifold can be locally isometrically embedded in a normed space. Thus a positive answer to the conjecture would allow to define the Holmes-Thompson valuations on arbitrary Finsler manifolds.
	
\subsection{Main results}
 In loose terms, the Finslerian setting allows an extra freedom compared with the Riemannian setting. Namely, we may vary not only the immersion, but also the ambient normed space itself, so long as the restricted norm is unchanged.

We refute Conjecture \ref{conj:fu} in general, while simultaneously proving that several weakened versions of the Weyl principle persist in the Finsler setting.  We find that in some settings, keeping the  immersion fixed, the restricted Holmes-Thompson intrinsic volumes are independent of the ambient norm. To give precise statements, let us introduce some terminology. All norms are assumed to have a smooth and positively curved unit ball.

\begin{definition}Let $M\subset V$ be an $m$-dimensional manifold without boundary, immersed in an $n$-dimensional linear space $V$. We say that $M$ satisfies the \emph{weak Weyl principle (WWP) at $p\in M$} if whenever $F, F'$ are two norms on $V$ that coincide on $\{T_xM:x\in U\}$ for some neighborhood $U$ of $p$, it also holds that $\mu^F_k|_U=\mu^{F'}_k|_U$, for all $k\geq 1$. We say that $M\subset V$ satisfies \emph{WWP} if it holds for all $p\in M$.

\end{definition}

Our first main result states that the weak Weyl principle holds unconditionally for small codimension, and furthermore holds generically when the codimension is not too large, as follows.
 \begin{MainTheorem}\label{thm:special_dimension}
	Fix a manifold $M$, assume $\dim M=m$ and $\dim V=n$. If $n\leq m+2$ then WWP holds for any immersion $e:M\looparrowright V$. If $m+3\leq n\leq 2m$, then WWP holds for a dense residual set of immersions $M\looparrowright V$ in the Whitney topology.
\end{MainTheorem}
Here by a residual set we mean the countable intersection of open dense sets.

Our second main theorem  shows that WWP fails for generic submanifolds of large codimension. We henceforth make use of an auxiliary Euclidean structure on $V=\RR^n$, and write $S(V)=S^{n-1}$ for the unit sphere.  We get an induced Riemannian structure on $M$, and define the tautological map $\theta:SM\to S^{n-1}$ by $\theta(x,v)=v$, where $SM$ denotes the unit tangent bundle.

\begin{definition}\label{def:directionally_regular}
	A manifold $M\subset V$ is \emph{directionally regular} at $(p,u)\in SM$ if the differential $d_{(p,u)}\theta\colon T_{(p,u)} SM \to T_u S^{n-1}$ has full rank, namely $\min(2m-1, n-1)$.
\end{definition}

Recall that the  \emph{$2$\textsuperscript{nd} osculating space} $\mathcal O_p^2M\subset T_p V$  at $p\in M$ is spanned by the velocity and acceleration vectors of all curves at $p$. The Euclidean orthogonal complement of $T_pM$ in $\mathcal O_p^2M$ is the \emph{$1$\textsuperscript{st} normal space} $\mathcal N_p^1M$ (see, e.g., \cite{spivak4}).
\begin{MainTheorem}\label{thm:false_in_general}
	Assume $m\geq 3$, and $M^m\subset V$ is directionally regular at $(p,u)\in SM$. Then WWP for $M$ fails at $p$ if $\dim \mathcal O^2_p M>2m$.
\end{MainTheorem}
Since the  $2$\textsuperscript{nd} osculating space is generically of dimension $m+\min(\frac{m(m+1)}{2}, \codim M)$, WWP generically fails when $\codim M>m\geq 3$. In particular, Conjecture \ref{conj:fu} fails for $\dim M\geq 3$.

Finally, we observe that a variant of the weak Weyl principle holds in general if we restrict not only the norm but also its derivatives, as follows.

\begin{MainTheorem}\label{thm:jet_weyl}
$\mu_k^F|_M$ is determined by the values $\frac{\partial^j F}{\partial v^j} (u)$, for all $(p,u)\in SM$ and $v\in \mathcal N^1_pM\subset T_u S^{n-1}$, for all $j\leq m-1$.
\end{MainTheorem}

	According to \cite{gu}, every $m$-dimensional Finsler manifold can be locally isometrically embedded in a $2m$-dimensional normed space, see also \cite{shen} and the references therein. In this situation, WWP holds by Theorem~\ref{thm:special_dimension} for a generic submanifold, but fails in general except in codimensions $1$ and $2$. 
	Moreover, we give an explicit example of a submanifold $M^3\subset \RR^6$ for which WWP fails (Example \ref{exm:M3R6}), which is the lowest dimensional setting when this can happen. Thus the genericity assumption in Theorem \ref{thm:special_dimension} is essential, and one cannot define Holmes-Thompson intrinsic volumes on general Finsler manifolds.

Let us comment on the choice of the Holmes-Thompson definition of volume for the generalization of the Euclidean intrinsic volumes to smooth Minkowski spaces.
Given a definition of volume $\calV$ for normed spaces, one may ask if  $k$-homogeneous valuations can be functorially assigned to a normed space $(V,F)$ for each $0\leq k\leq \dim V$, such that 
$\mu^\calV_k|_E= \calV_{(E,F|_E)}$  for each $k$-dimensional linear subspace $E\subset V$. By the Klain embedding theorem, $\mu_k^\calV$ would satisfy the linear Weyl principle.
According to J.C.~\'Alvarez Paiva \cite{alvarez_problems}, the three most commonly used definitions of volume in Finsler Geometry are: Busemann, Holmes-Thompson, and Benson (Gromov mass*). 
To date, only the Holmes-Thompson definition has been shown to be extendable as a family of valuations.  Moreover, the Holmes-Thompson intrinsic volumes form an algebra under the Alesker product \cite{bernig_HT}, and no other definition of volume would have this property. This follows at once by noting that all definitions of volume coincide for $1$-dimensional spaces, and by the Klain embedding theorem, the corresponding $1$-homogeneous valuation must coincide for all definitions of volume. This makes the Holmes-Thompson definition the most natural from the perspective of integral geometry (see also \cites{AlvarezFernandes:Crofton,SchneiderWieacker:IGMS,Ludwig:Areas} for further evidence).

 It is worth noting that the Weyl principle will be seen to fail already for the $1$-homogeneous valuations. Thus the Weyl principle would fail in the Finsler setting for any definition of volume which can be extended to a family of valuations.

\subsection{Plan of the paper.}
In Section~\ref{sec:background}, we recall the notions of valuation theory that we 
need and carry out the simple reduction to the 1-homegeneous case.
In Section~\ref{sec:current_computation}, we compute the defining current of the first Holmes-Thompson intrinsic volume. Then in Section~\ref{sec:fiber_integral}, we compute explicitly the restriction to a submanifold, which is given by certain fiber integrals. We then utilize classical harmonic analysis on the sphere to simplify those integrals.
Finally in Section~\ref{sec:main_proof} we utilize the previous computation to prove the main theorems, and construct some examples demonstrating the sharpness of our results. Section~\ref{sec:prolbems} is devoted to open problems and conjectures.

\section{Background and preparation}\label{sec:background}

\subsection{Some integral transforms}
We write $C^N_\epsilon(S^{n-1})$ for the space of $C^N$-smooth functions on the sphere  of parity $\epsilon=\pm$ with respect to the antipodal map.

We will use the following $\SO(n)$-equivariant integral operators:
\begin{enumerate}
	\item The cosine transform $\mathcal C:C^\infty(S^{n-1})\to C^\infty(S^{n-1})$ given by
	\[\mathcal C(f)(v)=\int_{S^{n-1}}|\langle u,v\rangle | f(u)du  \]
	\item The hemispherical transform $H:C^\infty(S^{n-1})\to C^\infty(S^{n-1})$ given by
	\[ H(f)(v)=\int_{u\in S^{n-1}: \langle u,v\rangle \geq 0}f(u) du \]
	
\end{enumerate}
When ambiguity can arise, we write $\mathcal C_n$ instead of $\mathcal C$.
It is clear that $\mathcal C$ vanishes on odd functions, and $\mathcal C$ is well-known to be invertible on the even functions, with a continuous inverse $\mathcal C^{-1}:C^\infty_+(S^{n-1})\to C^\infty_+(S^{n-1})$, see, e.g., \cite{groemer}.
In contrast, $H$ has rank one on the even functions, and is continuously invertible on the odd functions \cite{rubin}.

The $s$th Sobolev norm on $S^{m-1}$ is given by
$$\|f\|_{H^s}= \left(\sum_{{n\geq 0}} (1+n)^{2s} \|f_n\|^2_{L^2}\right)^{1/2},$$
where $f=\sum_{n}f_n$  is the decomposition into spherical harmonics, and the Sobolev space 
$H^s(S^{m-1})$ is the completion of $C^\infty(S^{m-1})$ with respect to this norm (see, e.g., \cite{Garrett}). The $C^k$-norm on $S^{m-1}$ is defined inductively by
$$ \|f\|_{C^k}= \|f\|_{C^0} + \sup_{i} \| X_i f\|_{C^{k-1}},$$
where $(X_i)$ is a basis of the Lie algebra $\mathfrak{so}(n)$ and $Xf(u)=\dt f(e^{tX}u)$ for each $X\in \mathfrak{so}(n)$. 
By \cite{groemer}*{Eq.~3.6.4}, there is a continuous embedding
\begin{equation}\label{eq:CkHs} C^{2k}(S^{m-1})\subset  H^s(S^{m-1}) \qquad \text{for } 2k> s+ \frac{n}{2};\end{equation}
By the Sobolev embedding theorem,
\begin{equation}\label{eq:sobolev} H^s(S^{m-1})\subset C^k(S^{n-1})\qquad \text{for } s> k+ \frac{n-1}{2}.\end{equation}
In particular, $\bigcap_{s>0} H^s(S^{m-1})=C^\infty(S^{m-1})$.

\subsection{Valuations} As a general reference for valuation theory we recommend  the books by  R.~Schneider \cite{Schneider:BM}*{Chapter 6} and S.~Alesker~\cite{Alesker:Book}.

We write $\Dens(V)$ for the one-dimensional space of Lebesgue measures on the finite-dimensional real vector space $V$. The even $k$-homogeneous valuations on $V$ are denoted $\Val_k^+(V)$.
The Klain map $\Klain:\Val_k^+(V)\to \Gamma(\Grass_k(V), \Dens(E))$ is given by $\Klain_\phi=\phi|_E$. The latter is a Lebesgue measure on $E$ by Hadwiger's theorem. By a theorem of Klain, $\Klain$ is injective.

Put $\mathbb P_M=\mathbb P_+(T^*M)$ for the co-sphere bundle of $M$, which is the oriented projectivization of the cotangent bundle. It is equipped with a natural contact structure, and we call a form $\omega\in \Omega(\mathbb P_M)$ \emph{vertical} if its restriction to any contact hyperplane vanishes.
Given a Riemannian structure on $M$, there is a natural choice of contact form $\alpha$ given by $\alpha_{x,\xi}=\langle \xi,d\pi(\bullet)\rangle$, where $\pi\colon \PP_M\to M$ is the natural projection, under the identification $\mathbb P_M=SM$.

Let us recall more precisely the notion of valuations on manifolds. For simplicity, we assume $M$ is oriented. The family of test bodies $\mathcal P(M)$ can be chosen in several ways; we take them to be the compact submanifolds with corners. For $X\in\mathcal P(M)$, its conormal cycle is $\nc(X)\subset \mathbb P_M$. It generalizes both the normal cycle of a convex set in a linear space and the conormal bundle of a submanifold, 
see, e.g. \cites{Fu:IGR, FPR:WDC, Fu:Subanalytic}

\begin{definition}
	The space of smooth valuations $\mathcal V^\infty(M)$ on an oriented manifold $M^m$ consists of functions $\phi:\mathcal P(M)\to \RR$ of the form
	\[ \phi(X)=\int_X\mu+\int_{\nc(X)}\omega \]
	where $\mu\in \Omega^m( M)$  and $\omega\in\Omega^{m-1}(\mathbb P_M)$.
\end{definition}

It was established in \cite{bernig-brocker} that a smooth valuation $\phi\in\mathcal V^\infty(M)$ is uniquely determined by its \emph{defining currents} $(C,T)\in  C^\infty(M)\times \Omega^{m}(\mathbb P_M)$, which can be any pair of forms subject to the conditions
\begin{enumerate}
	\item $T$ is closed and vertical.
	\item $\pi_*T=dC$, where $\pi:\mathbb P_M\to M$ is the projection.
\end{enumerate}
If $\phi\in \mathcal V^\infty(M)$ is represented by $(\omega,\mu)$ then $T=a^*(D\omega+\pi^*\mu)$, $C=\pi_*\omega$ where $a$ is the fiberwise antipodal map, and $D$ the Rumin differential operator. We write $\phi=[[\mu,\omega]]=[(C,T)]$, $T=T(\phi)$, $C=C(\phi)$.

Every immersion $e:M\looparrowright N$ of smooth manifolds induces a  pullback map  $e^*:\mathcal V^\infty(N)\to \mathcal V^\infty(M)$,  which for embeddings is given by
$$e^*\phi (X)= \phi(e(X)), \qquad X\in \calP(M).$$ 
 On the level of differential forms and in the special case $e:M\looparrowright V$
it can be  described as follows (see \cite{alesker_integral}). 
Let $N^*M$ denote the conormal bundle and consider the following natural maps:  the projection $q: \mathbb P_+(N^*M)\to M$; the fiberwise inclusion $\theta: \mathbb P_+(N^*M)\looparrowright \mathbb P_V$; the map
$\beta:\widetilde{ \mathbb P_V\times_V M}\to \mathbb P_M$, extending the projection $\mathbb P_V\times_V M \setminus N^*M\to \mathbb P_M$ to $\widetilde{ \mathbb P_V\times_V M}$, which is the oriented blow-up along $N^*M$; and $\bar\alpha:\widetilde{ \mathbb P_V\times_V M}\to \mathbb P_V$, the composition of the blow-up map with the natural map $ \mathbb P_V\times_V M\to \mathbb P_V$. 
If $\phi\in \mathcal V^\infty(V)$ is represented by $(\mu,\omega)\in  \Omega^n(M)\times\Omega^{m-1}(\mathbb P_M)$, 
then $e^*\phi=[[\mu', \omega']]$ where
\begin{equation}\label{eq:restriction_valuation}
\omega'=\beta_*\bar\alpha^*\omega, \quad \mu'=q_*\theta^*\omega.
\end{equation} 
Alternatively, if $\phi=[(C,T)]$, then $e^*\phi=[(C',T')]$ where
\begin{equation}\label{eq:restriction_valuation_CT}
T'=\beta_*\bar\alpha^*T, \quad C'=e^*C.
\end{equation}

\subsection{The trivial cases}

Here we verify Conjecture~\ref{conj:fu} for $\dim M\leq 2$ and show that in general it is enough to prove it for the first intrinsic volume. 
In the following, $V$ is a finite-dimensional real vector space.

\begin{lemma}
 Let $\varphi\in \Val^+_k(V)$ be a smooth valuation.  If $X\in \calP(V)$ is a compact $k$-dimensional submanifold with corners, then 
 $$\varphi(X)= \int_{X} \Klain_\varphi(T_xM).$$
\end{lemma}
\begin{proof}
 Similar to the proof of Lemma~3.7(ii) in \cite{SolanesWannerer:Spheres}.
\end{proof}

The following is a direct consequence of the above lemma and the definitions.

\begin{corollary}\label{cor:klain} Let $V$ be a normed space. 
 If $X\in \calP(V)$ is a compact $k$-dimensional manifold with corners, then 
 $\mu_k^F(X)$ equals the Holmes-Thompson volume of $(X, F|_X)$. If $M\subset V$ is a $k$-dimensional immersed submanifold, then $\mu_k^F|_M$ is the Holmes-Thompson volume measure of $(M, F|_M)$.
\end{corollary}

\begin{lemma}\label{lemma:topology}
 Consider a smooth valuation $\varphi\in\Val^+_{k}(V)$. If $X\subset V$ is a compact $(k+1)$-dimensional submanifold with boundary, then 
 $$\varphi(X) = \frac 1 2 \varphi(\partial X).$$
\end{lemma}
\begin{proof}
Note that $\phi$ has Euler-Verdier eigenvalue equal to $(-1)^k$. Now apply \cite{alesker4}*{Lemma 4.1.1}.
\end{proof}

\begin{corollary}\label{cor:codimension_one} Let $(V,F)$ be a normed space and let $M\subset V$ be an immersed smooth submanifold of dimension $k+1$. 
 Then the restriction of $\mu_k^F$ to $M$ is intrinsically defined. For every compact submanifold with corners $X\subset M$, $\mu_k^F(X)=\frac12 \vol^{HT}_{k}(\partial X, F|_{\partial X})$.
\end{corollary}
\begin{proof}
 For $X$ a submanifold with boundary, this follows from Lemma~\ref{lemma:topology} and Corollary~\ref{cor:klain}. 
 As  the proof of \cite{bernig-brocker}*{Theorem 1} shows that a smooth valuation is determined by its values on smooth submanifolds with boundary, we conclude that  $\mu_k^F|_M$ is intrinsically defined. For the last statement, we may use continuity and approximate a manifold with corners by a sequence of manifolds with boundary.
\end{proof}

In particular, the Weyl principle holds for Finsler surfaces.

\begin{theorem}\label{thm:dim2}
 Conjecture~\ref{conj:fu} holds for $\dim M\leq 2$. 
\end{theorem}
\begin{proof}
 This is an immediate consequence of Corollaries~\ref{cor:klain} and \ref{cor:codimension_one}.
\end{proof}

Finally, we  observe that it suffices to prove the Weyl principle for the first intrinsic volume.

\begin{proposition}\label{prop:k=1_is_enough}
	Let $M\subset V$ be a submanifold. If $\mu^F_1|_M$ is intrinsically defined, then so are all $\mu_k^F|_M$ for all $k\geq 2$.
\end{proposition}
\begin{proof}
	This follows at once by applying the Alesker product. It holds that $\mu_k^F=c_k(\mu_1^F)^k$ for certain universal coefficients $c_k$ \cite{bernig_HT}, and it remains to recall that the Alesker product commutes with restrictions.
	\end{proof}

\section{The defining current of the first intrinsic volume} \label{sec:current_computation}
For the remainder of the paper, we fix an auxiliary Euclidean structure on $V=\RR^n$, and let $S^{n-1}$ be the corresponding unit sphere.

\begin{lemma}\label{lem:Tmu1} It holds that 
\begin{equation}\label{eq:linear_current}T(\mu_1^F)= c_n\mathcal  C^{-1}(F)  \, \alpha \wedge  \vol_{S^{n-1}}
, \qquad C(\mu_1^F)=0,\end{equation} where  $c_n$ is a constant depending only on $n$, and $\alpha$ is the contact form.
\end{lemma}

\begin{proof} Write the cosine transform in a $\mathrm{GL}(V)$-equivariant form (see \cite{alesker_bernstein}):
	$$\mathcal C\colon C^\infty (\Grass_{n-1}(V), \Dens(V/E)\otimes |\omega|) \to C^\infty( \Grass_1(V), \Dens(E)),$$
where $|\omega|=\Dens(T_E\Grass_{n-1}(V))$ is the line bundle of smooth densities on $\Grass_{n-1}(V)$. 
	As the Klain section of $\mu_1^F$ is $F$, it holds that $$\mu_1^F(K)=\int_{\AGrass_{n-1}(V)} \chi(K\cap E)dm(E)$$ for $K\in\mathcal K(V)$, where $m=\mathcal C^{-1}F$ is the Crofton measure of $\mu_1^F$, and we use the natural identification between smooth, translation-invariant measures on the affine Grassmannian $\AGrass_{n-1}(V)$ with $C^\infty (\Grass_{n-1}(V), \Dens(V/E)\otimes |\omega|)$. 
	
	Recall that the generalized valuations $\mathcal V^{-\infty}(V)$ are the dual space to the smooth, compactly supported valuations $\mathcal V_c^\infty(V)$ equipped with the weak topology, see \cite{alesker4}. The representation of valuations through its pair of defining currents $(C,T)$ extends naturally to generalized valuations, except $C$ is now a distribution and $T$ a current, see \cite{alesker_bernig}. Note that any (not necessarily compact) submanifold with corners $X\subset V$ defines an element $\chi_X\in\mathcal V^{-\infty}(V)$ through the evaluation map, and has $C(\chi_X)=\mathbbm 1_X$, $T(X)= \nc(X)$. 
	
	By Alesker-Poincare duality we find that for any $\psi \in \mathcal V^\infty_c(V)$, $$\langle \mu_1^F,\psi\rangle=\int_{\AGrass_1(V)} \psi(E) dm(E).$$ That is, $\mu_1^F=\int_{\AGrass_1(V)} \chi_E dm(E)$, implying \[T(\mu^F_1)=\int_{\AGrass_1(V)} T(\chi_E)dm(E)=\int_{\AGrass_1(V)} \nc(E)dm(E).\]
Using the Euclidean structure to identify $m = C^{-1} F$ with an element of $C^\infty(S^{n-1})$, 
and noting that $ \nc(u^\perp+tu) = \lcur (u^\perp+tu) \times \{\pm u\} \rcur$,
we conclude that 
  \[\langle T(\mu_1^F),\eta \rangle =  c_n \int_{\RR^n\times S^{n-1}}   C^{-1}(F)      \, \alpha \wedge  \vol_{S^{n-1}} \wedge \eta \]
 for every  $\eta\in \Omega^{n-1}_c(\RR^n\times S^{n-1})$.
\end{proof}

\section{Computing the restriction}\label{sec:fiber_integral}

\subsection{The fiber integrals}\label{subsec:coefficients_integral_expressions}
Let $M\subset (V,F)$ be an immersed submanifold of dimension $m$.  Recall $V=\RR^n$ has a fixed Euclidean structure. 

Fix a point $p\in M$ and local coordinates $(x^0, x^1,\dots,x^{m-1})$ for $M$ such that
$$ \left.\partial_{x^i}\right|_p = e_i, \qquad i=0,\dots,m-1$$
for some  orthonormal basis  $e_0,e_1,\ldots, e_{n-1}$ of $V$.

Let us fix also a point $(p,u)\in SM$. Without loss of generality we may assume that $u=e_0$. We denote by ${\widetilde\partial _{x^i}}$ the horizontal lift of $ \partial _{x^i}$ to $SM$ with respect to the induced Riemannian structure.  
Note that  
\begin{equation}\label{eq:basis4}{\widetilde\partial _{x^0}}, \dots, {\widetilde \partial _{x^{m-1}}}, e_1, \dots,e_{m-1}\end{equation}
form a basis of $T_{(p,u)} SM$. To construct these lifts, let 
$\gamma_i\colon I\to M$ be a smooth curve with $\gamma_i(0)=p$ and $\gamma_i'(0)= \frac{\partial}{\partial x^i}$. 
Then 
$$ { \widetilde\partial _{x^i}} = \dt P_\gamma u(t),$$
where $P_\gamma u$ is the curve in $SM$ defined via parallel transport of $u$ along $\gamma$. 

We will also need the second fundamental form $h\colon T_pM\otimes T_pM\to T_p M^\perp$ of $M$, 
$$h(X,Y)= (\overline\nabla_XY )^\perp,$$
where $\overline\nabla$ denotes the standard connection in  $\RR^n$.  We write 
$h_{ij}= h(e_i,e_j)$, and $h_{ij}^N=\langle h_{ij},N\rangle$.

Let $P^k_{m-1}$ be the set of increasing functions $\sigma:\{1,\dots,k\}\to \{1,\dots,m-1\}$. 
The complement is the obvious map $P^{k}_{m-1}\to  P^{m-1-k}_{m-1}$, $\sigma\mapsto\sigma^c$.

Let $\theta_0,\dots,\theta_{m-1},\omega_{1,0},\dots,\omega_{m-1,0}\in T_{(p,u)}^*SM$ denote the basis dual to \eqref{eq:basis4}. By \eqref{eq:restriction_valuation_CT}, the restrictred valuation $\mu^F_1|_M$ has $T_M^F:=T(\mu^F_1|_M)=\beta_*\bar\alpha^* T(\mu_1^F)$. As it is a vertical $m$-form, there are  coefficients $A^k_{\sigma\tau}$ with $0\leq k\leq m-1$, $\sigma\in P^k_{m-1}$, $\tau\in P^{m-1-k}_{m-1}$ such that
\begin{equation}\label{eq:restriction_current}T_M^F =  \theta_0\wedge \sum_{k=0}^{m-1}\sum_{\sigma\in P^k_{m-1}}\sum_{\tau\in P^{m-1-k}_{m-1}} A^k_{\sigma\tau}\theta_{\sigma (1)}\wedge\dots\wedge\theta_{\sigma (k)}\wedge \omega_{\tau(1),0}\wedge\dots\wedge\omega_{\tau (m-1-k),0}.\end{equation}

 \begin{remark}
		More generally, the coefficients $A^k_{\sigma\tau}$ are naturally functions on the orthonormal frame bundle of $M$, that is, they are determined by the fixed  immersion $M\looparrowright V$, the norm $F$, and the auxiliary Euclidean structure on $V$, which induces a Riemannian structure on $M$.
	\end{remark}
\begin{proposition} It holds for all $k,\sigma,\tau$ that
	\begin{align}\label{eq:general_coefficient}  &A^k_{\sigma\tau}=\epsilon c_n \int_{0}^{\pi/2} \int_{S(T_pM^\perp)}\det h_{\sigma\tau^c}^N(p) \cos^{m-k}\phi  \sin^{n-m-1+k} \phi   \;  \mathcal C^{-1}F(\cos\phi u + \sin\phi N)  \; dN \; d\phi,
	\end{align}
where $c_n$ is as in \eqref{eq:linear_current}, and $\epsilon\in\{\pm1\}$ depends explicitly on $\sigma,\tau$.
\end{proposition}
\begin{proof}
	As the form $T_M^F$ is smooth, we may ignore the blow-up construction of the restriction.
	Thus by Lemma \ref{lem:Tmu1}, we have to  compute the pushforward
	of $ \mathcal C^{-1}F\, \alpha \wedge  \vol_{S^{n-1}}$ under the map $\beta\colon SV|_M\setminus TM^\perp\to SM$.
		
	Choose a smooth local section $N$ of the normal sphere bundle $S(TM^\perp)$.
	We define elements of $T_{p,\cos \phi u + \sin\phi N_p} SV|_M$ by
	\begin{align*} {\widehat\partial_{x^i}} & = \dt(\gamma_i(t),  \cos\phi P_{\gamma_i} u + \sin \phi N\circ\gamma_i), \qquad i=0,\dots,m-1\\ 
		\hat e_i & = \dt (p, \cos\phi(\cos t u + \sin t e_i) + \sin \phi N) = (0, \cos\phi e_i),\qquad i=1,\dots,m-1,
	\end{align*}
where $\gamma_i\colon I\to M$ is a smooth curve with $\gamma_i(0)=p$ and $\gamma_i'(0)= \frac{\partial}{\partial x^i}$.	
	We may express the second coordinate $v_i$ of $ {\widehat\partial_{x^i}}=(e_i,v_i)$ explicitly in terms of the second fundamental form, namely
	\begin{equation}\label{eq:comp_lift} v_i = \cos\phi h_{0i} + \sin\phi \left(  -\sum_{j=0}^{m-1}h_{ij}^N e_j + \nabla^\perp _{e_i} N\right)
	\end{equation}
	where $\nabla^\perp$ denotes the normal connection.
 	Indeed, if $\overline \nabla$ denotes the standard connection of the ambient space $V$, then for $0\leq i\leq m-1$
	$$\dt \langle P_{\gamma_i} u, e_j\rangle = \langle \overline\nabla_{e_i} P_{\gamma_i} u, e_j\rangle 
	=\langle \nabla_{e_i} P_{\gamma_i} u, e_j\rangle =0, \qquad j=0,\dots,m-1,$$
since for this range of indices,  $e_j$ is tangent to $M$ at $p$. Since $e_m,\ldots,e_{n-1}\in T_pM^\perp$, by the definition of the second fundamental form, we have 
	$$\dt \langle P_{\gamma_i} u, e_j\rangle = \langle \overline\nabla_{e_i} P_{\gamma_i} u, e_j\rangle 
	=  \langle h_{0i},e_j\rangle , \qquad j=m,\ldots, n-1.$$

	To compute the constants $A^k_{\sigma\tau}$, we first note that 
	$$e_1,\dots,e_{m-1}, -\sin\phi u + \cos\phi N, w_1,\ldots, w_{n-m-1}$$ is an orthonormal basis of $ T_{\cos\phi u + \sin\phi N} S^{n-1}$ if $N,w_1,\ldots, w_{n-m-1}$ is an orthonormal basis of $T_pM^\perp$, and  
	$$ \alpha( {\widehat\partial_{x^i}}) = \langle \cos\phi u+\sin\phi N, e_i\rangle= \begin{cases}
	\cos\phi, &  i=0,\\
	0, & i=1,\dots,m-1                                                            \end{cases}$$
	
	We integrate along the fiber, parametrized by $\gamma(\phi,N)=(p, \cos\phi u+\sin \phi N)$ so that $\frac{\partial}{\partial\phi}\gamma(\phi,N)=-\sin\phi u +\cos\phi N$.
	First, we compute
	
	\begin{align*}  &B(\phi,N):=\alpha \wedge  \vol_{S^{n-1}}({\widehat\partial_{x^0}}, (\widehat\partial_{x^{\sigma(j)}})_{j=1}^k,(\hat e_{\tau(j)})_{j=1}^{m-1-k}, \pder{\phi}\gamma(\phi,N),  w_1,\ldots,  w_{n-m-1}) 
\\&=\cos\phi\cdot \cos^{m-1-k}\phi \cdot \vol_{S^{m-1}}((v_{\sigma( j)})_{j=1}^k,( e_{\tau(j)})_{j=1}^{m-1-k})
\\&=(-1)^k\cos^{m-k}\phi\sin^k\phi\cdot \vol_{S^{m-1}}( \sum_{j=0}^{m-1} h^N_{\sigma( 1),j}e_j,\dots,\sum_{j=0}^{m-1} h^N_{\sigma (k),j}  e_j, (e_{\tau( j)})_{j=1}^{m-1-k})
\\&=\epsilon \cos^{m-k}\phi\sin^k \phi\det h^N_{\sigma\tau^c} ,
 \end{align*}
for some $\epsilon\in\{\pm1\}$.
	 This yields
		\begin{align*}A^k_{\sigma\tau} & =c_n\beta_*\bar\alpha^*( \mathcal C^{-1}F\cdot \alpha \wedge  \vol_{S^{n-1}} )({\widetilde\partial_{x^0}}, {\widetilde \partial _{x^{\sigma( 1)}}},\dots,{\widetilde\partial_{x^{\sigma( k)}}},  e_{\tau(1)},\dots, e_{\tau(m-k-1)})\\
		&= c_n \int_{0}^{\pi/2} \int_{S(T_pM^\perp)}  \sin^{n-m-1} \phi B(\phi,N) \;  \mathcal C^{-1}F(\cos\phi u + \sin\phi N)\;dN\; d\phi  \\
		& = \epsilon c_n \int_{0}^{\pi/2} \int_{S(T_pM^\perp)}\det h_{\sigma\tau^c}^N(p) \cos^{m-k}\phi  \sin^{n-m-1+k} \phi   \;  \mathcal C^{-1}F(\cos\phi u + \sin\phi N)  \; dN \; d\phi.
	\end{align*}
\end{proof}
The exact constant $c_n$ and sign $\epsilon$ will play no role in the following, and we omit them.
In particular, we write
	
\begin{align*} A^0 = \int_{0}^{\pi/2} \int_{S(T_pM^\perp)}  \cos^m  \phi\sin^{n-m-1}\phi  \;  \mathcal C^{-1}F(\cos\phi u + \sin\phi N) \; dN \; d\phi,
\end{align*}

\begin{align*} A^1_{\sigma\tau} =  \langle h_{\sigma(1),\tau^c}, \int_{0}^{\pi/2} \int_{S(T_pM^\perp)} \cos^{m-1}\phi  \sin^{n-m} \phi    N    \mathcal C^{-1}F(\cos\phi u + \sin\phi N)  dN  d\phi \rangle .
\end{align*}

\subsection{Simplifying the fiber integrals}
Fix $S^{m-1}\subset S^{n-1}$. For almost every $v\in S^{n-1}$ we may write  $v= \cos \phi u + \sin\phi N$ with $u\in S^{m-1}$, $N\in S^{n-m-1}$, $\phi\in[0,\frac\pi 2]$ in a unique way.
For $f\in C^\infty(S^{n-1})$, $u\in S^{m-1}$ and $N\in S^{n-m-1}$ we will write \[ \frac{\partial ^kf }{\partial N^k}(u):=\left.\frac{d^k}{d\phi^k}\right|_{\phi=0}f(\cos\phi u+\sin\phi N).\]
More generally, if $\tilde N=\lambda N$ with $\lambda\in\RR$, we write 
\begin{equation}\label{eq:directional_derivative} \frac{\partial ^k f}{\partial \tilde N^k}(u)=\lambda^k\frac{\partial ^kf }{\partial  N^k}(u).\end{equation}

The coefficient $A^0$ is straightforward to simplify.
\begin{lemma}\label{lem:a0_coeff}
	$A^0|_{S_pM}=\mathcal C_m^{-1}(F|_{S_pM})$.
	\end{lemma}
\begin{proof}
Take $u'\in S_pM$ and compute
\begin{align*}& \mathcal C_mA^0(u')\\&=\int_0^{\pi/2} \int_{S_pM}\int_{S(T_pM^\perp)} |\langle u',u\rangle|\cos^m\phi\sin^{n-m-1}\phi \mathcal C_n^{-1}F(\cos \phi u+\sin\phi N) dN dud\phi 
\\&= \int_0^{\pi/2}\int_{S_pM}\int_{S(T_pM^\perp)}\cos^{m-1}\phi\sin^{n-m-1}\phi|\langle u',v\rangle|\mathcal C_n^{-1}F(v)  dN dud\phi, 
\end{align*}
where $v=\cos \phi u+\sin\phi N$ is the parametrization $v:  S^{m-1}\times [0,\pi/2]\times  S^{n-m-1}\to S^{n-1}$.
Since $dv = \cos^{m-1}\phi\sin^{n-m-1}\phi  dN dud\phi$, we conclude that 
\[\mathcal C_mA^0(u') = \mathcal C_n(\mathcal C_n^{-1}F)(u')=F(u').\qedhere \]

	\end{proof}

\begin{corollary}\label{cor:a0_intrinsic}
	The coefficient $A^0$ is determined by $F|_M$.
\end{corollary}

As a warm-up for the general case, we compute $A^1_{\sigma\tau}$. We write $\partial_\phi|_{\phi_0}$ for $\left.\frac{\partial}{\partial \phi}\right|_{\phi=\phi_0}$.
\begin{lemma}\label{lem:B_integral}
 For every even smooth function $f$ on $S^{n-1}$, $u\in S^{m-1}$, and $N\in S^{n-m-1}$ we have 
 \[
  \int_{0}^{\pi/2} \int_{S^{n-m-1}}  \cos^{m-1}\phi'  \sin^{n-m} \phi' \langle N,N'\rangle  f (\cos \phi u +\sin \phi N')
  \; d N' \; d\phi'\]
  \[  = \frac 12  H^{-1}(  u'\mapsto \partial_\phi|_{0} \mathcal Cf(u',\phi, N) )(u),\]
  where here and in the  following $(u,\phi, N)$ denotes the point $\cos\phi u + \sin\phi N\in S^{n-1}$. 
\end{lemma}

\begin{proof}
 The cosine transform can be written as 
 \begin{align*} &\mathcal Cf(u,\phi,N)\\&= \int_{S^{m-1}}\int_{S^{n-m-1}}\int_{0}^{\pi/2} \cos^{m-1}\phi' \sin^{n-m-1}\phi' 
 |K(u',\phi',N')| f(u',\phi',N') \; d\phi' dN' du',\end{align*}
 where 
 $$K(u',\phi',N'):= K(u,\phi,N,u',\phi',N')  :=  \cos\phi \cos\phi' \langle u,u'\rangle + 
 \sin\phi \sin\phi' \langle N,N'\rangle .$$
 Put $g_{\phi, u, N}(u',\phi',N')=\cos^{m-1}\phi' \sin^{n-m-1}\phi' 
 K(u,\phi,N,u',\phi',N') f(u',\phi',N')$. We typically have $u, N$ fixed, and omit them from the notation.
Let us  compute 
 \begin{align*} \partial_\phi|_{0}  \mathcal Cf(u,\phi,N)&= \partial_{\phi}|_{0}
 \left( \int_{(u',\phi',N'):K(\phi)\geq 0 } g_{\phi} -\int_{(u',\phi',N'):K(\phi)\leq 0 } g_{\phi}\right)\\ 
 &= 2\partial_{\phi}|_{0}
  \int_{(u',\phi',N'):K(\phi)\geq 0 } g_{\phi}.\end{align*}
  where the last  equality follows by the symmetry  $(u',\phi',N')\mapsto (-u',\phi',-N')$, and the assumption that $f$  be even.

 For any distribution $\eta$ on the sphere, $\langle \eta, g_\phi\rangle $ is by \cite{Hormander:ALPDO}*{Theorem 2.1.3} 
 a smooth  function of $\phi$ and $\partial_\phi \langle  \eta, g_\phi\rangle =\langle  \eta, \partial_\phi  g_\phi\rangle$.  The indicator function $\mathbbm 1_{K(\phi)\geq 0}$ regarded as a distribution on  $S^{n-1}$ depending on the parameter $\phi$ is  differentiable at $\phi=0$ and its derivative is  a measure supported on $\{(u',\phi',N'): K(0)=0\}$,  see Eq.~\eqref{eq:der_phi_pos} below.    By the chain rule, 
 \begin{equation}\label{eq:der_phi2}\partial_{\phi}|_{0}
 \int_{K(\phi)\geq 0 } g_\phi =   \partial_{\phi}|_{0} \langle g_\phi,\mathbbm 1_{K(\phi)\geq 0}\rangle
 =\langle \partial_{\phi}|_{0}g_\phi, \mathbbm 1_{K(0)\geq 0}\rangle + \langle g_{0}, \partial_{\phi}|_{0} \mathbbm 1_{K(\phi)\geq 0}\rangle
 \end{equation}

Since $g_0$ vanishes whenever $K(0)$ vanishes, the second summand is zero.

 We conclude that 
 \begin{align*}\partial_\phi|_0 & \mathcal   Cf(u,\phi,N) =  2\int_{K(0)\geq 0 } \partial_\phi|_0  g_\phi\\
 & = \int_{\langle u,u'\rangle\geq 0 } \int_{S^{n-m-1}}\int_{0}^{\pi/2} \cos^{m-1}\phi' \sin^{n-m}\phi' 
 \langle N,N'\rangle f(u',\phi',N') \; d\phi' dN' du'\\
 &= 2 H_{u'}\left(\int_{S^{n-m-1}}\int_{0}^{\pi/2} \cos^{m-1}\phi' \sin^{n-m}\phi' 
 \langle N,N'\rangle f(u',\phi',N') \; d\phi' dN'  \right)(u), \end{align*} where $H_{u'}$ denotes the hemispherical transform with respect to $u'\in S^{m-1}$.
 Seeing that the argument of the hemispherical transform is odd in $u'$, we may apply the inverse to conclude the proof.
\end{proof}

Putting $\tilde N=h_{\sigma(1),\tau^c}$ and recalling $u=e_0$, we conclude, using notation \eqref{eq:directional_derivative}, that 
\begin{equation}\label{eq:a1_coefficient}A^1_{\sigma\tau}(u)=\frac12 H^{-1}(u'\mapsto \frac{\partial F}{\partial h_{\sigma(1),\tau^c}} (u')) (u)\end{equation}

For the general coefficient we do not obtain an explicit expression, but only some qualitative information. We first prove a lemma in harmonic analysis on the sphere.
\begin{definition}
	We define the $\SO(m)$-equivariant transforms \[R_{D}^k, C_D^k:C^\infty(S^{m-1})\to C^{\infty}(S^{m-1})\]  by 
	\[R_D^kf(u)=\int_{u'\perp u} \left.\frac{\partial^k}{\partial t^k}\right|_{t=0} f(\sqrt{1-t^2} u'+t u)du', \qquad C_D^kf(u)=R_D^k(\langle u, \,\bullet \, \rangle f)(u) \]
\end{definition}
\begin{lemma}\label{lem:alternating_sum_invertible}
	Assume $m\geq 3$. Let $a_{k}, a_{k-2},\dots, a_\epsilon\geq 0$ with $a_\epsilon\neq 0$,  $\epsilon\equiv k\mod 2\in\{0,1\}$. Then
	\[T:=a_{k}R_D^k-a_{k-2}R_D^{k-2}+\dots +(-1)^{\lfloor k /2\rfloor} a_\epsilon R_D^\epsilon\] is a bijection on the space $C^\infty_\epsilon(S^{m-1})$ of smooth functions on $S^{m-1}$ of parity $\epsilon$, and $T^{-1}$ is continuous.
	More precisely, there is an explicit constant $K=K(m)>0$ such that for every $q\geq0$, the operator $T^{-1}:C^{q+K}_\epsilon(S^{m-1})\to C^q_\epsilon(S^{m-1})$ is continuous.
\end{lemma}
\begin{proof}

As the transform commutes with $\textrm{SO}(m)$, $R_D^k$ is given by a scalar on the space of spherical harmonics of order $n$, denoted $\mathcal H^m_n\subset C^\infty(S^{m-1})$. Let us first check that the eigenvalues of $R_D^k$, $R_D^{k+2}$ on $\mathcal H^m_n$ have opposite signs for every $k\geq 0$ of parity $\epsilon$, which would immediately imply the injectivity of $T$.
Evaluating both on a zonal harmonic given by the Ledgendre polynomial $P_n^m(\langle u', e_m\rangle)$ with $n\equiv \epsilon\mod 2$ and for $u=e_m$, we see that we  ought to check

\[(P^m_n)^{(k)}(0)(P^m_n)^{(k+2)}(0)<0 \]

It holds for $n\geq r$ that $\frac{d^r}{dt^r}P^m_{n}(t)=cP_{n-r}^{m+2r}(t)$ for some $c>0$ \cite{groemer}*{Lemma 3.3.9}. Taking $r=k,k+2$ and noting that $n'=n-k$ is even, it remains to verify that for any $m'$ and even $n'$ we have
\[P_{n'}^{m'}(0)P_{n'-2}^{m'+4}(0)<0. \]
Now $\sign P_{n'}^{m'}(0)=(-1)^{n'/2}$ according to \cite{groemer}*{Eq.~(3.3.20)}, and correspondingly $\sign P_{n'-2}^{m'+4}(0)=(-1)^{-1+n'/2}$. Thus $T$ is non-zero on every harmonic of parity $\epsilon$ and hence injective.

To prove that $T$ is surjective, let $\lambda_n$ be the eigenvalue of $T$ on $\mathcal H^m_n$, and $r_{n}$ the eigenvalue of $R_D^{ \epsilon}$ on $\mathcal H^m_n$. 
Recall from \cite{groemer}*{Eq.~3.1.15} that $\dim \mathcal H_n^m=:N(m,n)=O(n^{m-2})$ as $n\to \infty$. Letting $c_j$ denote various constants  independent of $n$, 
we have by  \cite{groemer}*{Eq.~(3.4.19)}  for  $l$ even 
$|P^d_l(0)| > c_0 l^{-\frac{d-2}{2}}$ 
and hence by \cite{groemer}*{Lemma 3.3.9} 
\[ |r_n|=c_1 |(P_n^m)^{(\epsilon)}(0)|=c_2\frac{N(m+2 \epsilon,n)}{N(m,n+ \epsilon)} |P_{n- \epsilon}^{m+2 \epsilon}(0)|\geq c_3  n^{-m/2 + \epsilon +1}\]
As the signs in the definition of $T$  are alternating, it holds that $|\lambda_n|\geq |a_\epsilon||r_n|$ and thus 
\begin{equation}\label{eq:growth} |1/ \lambda_n| \leq c_4 n^{m/2 -\epsilon -1}.\end{equation}

Take $g\in C^{\infty}_\epsilon(S^{m-1})$, and let $g=\sum_{n\equiv \epsilon}g_n$ be its decomposition into spherical harmonics.
Set $f:=\sum_{n\equiv\epsilon}\frac{1}{\lambda_n} g_n$. Then  $f\in \bigcap_{s>0} H^s(S^{m-1})=C^\infty(S^{m-1})$ by  \eqref{eq:growth} and $Tf=g$. This proves that $T$ is surjective. 
Moreover, by the Sobolev embedding theorem \eqref{eq:sobolev}, \eqref{eq:growth}, and \eqref{eq:CkHs}  we can for every $q\geq 0$ find $s,q'=q+K(m)$ such that \[\|f\|_{C^q}\leq c_5\|g\|_{H^{s}} \leq c_6\|g\|_{C^{q'}} . \] 
This proves the continuity of $T^{-1}$ and concludes  the proof.
\end{proof}

\begin{proposition}\label{prop:fiber_simiplification}
	For an even function $f\in C^\infty_+(S^{n-1})$, point $u\in S^{m-1}$, and normal vector $N\in S^{n-m-1}$, define
	\begin{align} \label{eq:Inm}
	\begin{split}&	
	 	 I_k^{n,m}(u,N;f):=\\&\int_{0}^{\pi/2} \int_{S^{n-m-1}}  \cos^{m-k}\phi'  \sin^{n-m-1+k} \phi' \langle N,N'\rangle^k  f(\cos\phi u+\sin \phi N')  
	\; d N' \; d\phi' .
	\end{split}
	\end{align}
	Then for any $q\geq 0$, \[\|I_k^{n,m}(\bullet,N;f)\|_{C^q(S^{m-1})}\leq c \sum_{j=0}^k \left\| \frac{\partial ^j }{\partial N^j}\mathcal C_n(f) \right\|_{C^L(S^{m-1})} \] for some constants $c=c(n,m,k,q)>0$, and $L=L(n,m,k,q)\in\mathbb N$.
\end{proposition}
\begin{proof}

	As in the proof of Lemma~\ref{lem:B_integral} we write  $$ K(u,\phi,N,u',\phi',N')  =  \cos\phi \cos\phi' \langle u,u'\rangle + 
	\sin\phi \sin\phi' \langle N,N'\rangle, $$ \[g_\phi(u',\phi',N')=\cos^{m-1}\phi' \sin^{n-m-1}\phi' 
	K(u,\phi,N,u',\phi',N') f(u',\phi',N').\]
	We note that $\partial^2_\phi g_\phi=-g_\phi$, and record
	\[g_0=\cos^m\phi'\sin^{n-m-1}\phi' \langle u,u'\rangle f(u',\phi',N')\]
	\[g_0':=\partial_\phi|_{0}g_\phi=\cos^{m-1}\phi'\sin^{n-m}\phi'\langle N,N'\rangle f(u',\phi',N')\]

As in the proof of Lemma~\ref{lem:B_integral} we view  $\mathbbm 1_{K(\phi)\geq 0}$ as a distribution on  $S^{n-1}$ depending on the parameter $\phi$. The first and  main technical step in the proof is to compute $\partial ^l_\phi \mathbbm 1_{K(\phi)\geq0}$ for any $l>0$.	Note that the  condition $K(\phi)\geq 0$ is equivalent to 
$$\langle u,u'\rangle \geq  -\tan\phi \tan \phi' \langle N, N' \rangle=:a(\phi)=a_{\phi', N'}(\phi).$$
Putting $S^{m-2}_u=u^\perp\cap S^{m-1}$, we have for any function $h\in C^\infty(S^{m-1})$
$$\int_{\langle u,u'\rangle \geq a(\phi)} h(u')\; du'  = \int_{ a(\phi)}^1 \int_{S^{m-2}_u}   h(\sqrt{1-t^2} v ' + t u) \sqrt{1-t^2}^{m-3}\; dv' \; dt.$$
Since for any $\tilde h\in C^\infty(\RR)$
$$ \partial_\phi \int_{a(\phi)}^1 \tilde h(t)\; dt = -\tilde h(a(\phi)) a'(\phi),$$
we conclude that for any smooth function $h(u',\phi',N')$ on $S^{n-1}$ parametrized by $S^{m-1}\times[0,\frac \pi 2]\times S^{n-m-1}$,
\begin{align}\label{eq:der_phi_pos}&  \nonumber\partial_\phi  \int_{u',\phi',N'\colon K(\phi)\geq 0} h(u',\phi',N')  \\&= 
\frac{ 1 }{ \cos^2 \phi} \int_{u',\phi',N'\colon a_\phi(\phi',N')=\langle u,u'\rangle} \tan\phi' \langle N, N'\rangle h(u',\phi',N')
\\&=\frac{ 1 }{ \cos^2 \phi} \int_{v'\in S_u^{m-2},\phi',N'} \tan\phi' \langle N, N'\rangle h(\cos \theta v'+\sin\theta u,\phi',N')\sqrt{1-a(\phi)^2}^{m-3},\nonumber
\end{align}
where we use the auxiliary function $\theta=\theta(\phi',N')$ given by $\sin\theta=a_{\phi',N'}(\phi)$.

Introduce \begin{align*}J^{\nu}_{r,s, t}(u,\phi,N;h):=\tan^t\phi\int_{v',\phi',N'}  \sqrt{1-a_{\phi',N'}(\phi)^2}^s & \langle N,N'\rangle^\nu\tan^\nu\phi'\cdot \\ & \cdot\partial_ z^r|_{a(\phi)}h(\sqrt{1-z^2} v'+z u,\phi',N')dv'dN'd\phi'.\end{align*}
Putting $\phi=0$, we find
\[ J^{\nu}_{r,s, t}|_{\phi=0}=\left\{\begin{array}{cc}J^{\nu}_{r}(u,N;h),& t=0\\0,& t>0\end{array}\right. \]
where \begin{align*}J^{\nu}_{r}(u,N; h)&:=\int_{ u'\perp u,\phi',N'}\langle N, N'\rangle ^\nu\tan^\nu \phi'\;\partial^r_z|_0h(\sqrt{1-z^2} u'+z u) \;du'd\phi'dN' .\end{align*}
It follows  directly from the definitions  that
\begin{align}\label{eq:particular_h}
\begin{split}
J^k_r(u,N;g_0)&=C_D^r[I_k^{n,m}(u',N;f)];\\  
 J^{k-1}_r(u,N; g_0')&=R_D^r[I_k^{n,m}(u',N;f)].
 \end{split}
\end{align}

By eq.~\eqref{eq:der_phi_pos},  we have 
\begin{equation}\label{eq:start_step}\langle \partial_\phi \mathbbm 1_{K(\phi)\geq 0}, h\rangle=J^{1}_{0, m-3, 0}+J^{1}_{0, m-3, 2} \end{equation}
and it is easy to verify that
\begin{equation}\label{eq:long_recursion}\partial_\phi J^{\nu}_{r,s,t}=t(J_{r,s,t-1}^{\nu} + J_{r,s,t+1}^{\nu})-s(J_{r,s-2,t+1}^{\nu+2}+J_{r,s-2,t+3}^{\nu+2})-(J_{r+1,s,t}^{\nu+1}+J_{r+1,s,t+2}^{\nu+1}) \end{equation}
The value of $\langle \partial_\phi^l|_{0}  \mathbbm 1_{K(\phi)\geq 0}, h\rangle$ will thus be a linear combination of various $J^\nu_r(u,N;h)$.
We think of the end formula as the result of $(l-1)$ differentiations applied to each of the two summands in \eqref{eq:start_step}, the sum at each step consisting of summands $J^\nu_{r,s,t}$, splitting into six summands according to \eqref{eq:long_recursion}. In the last step, all summands with $t>0$ vanish. Each sequence of splittings is referred to as a path. The combinatorics is illustrated in the following diagram.

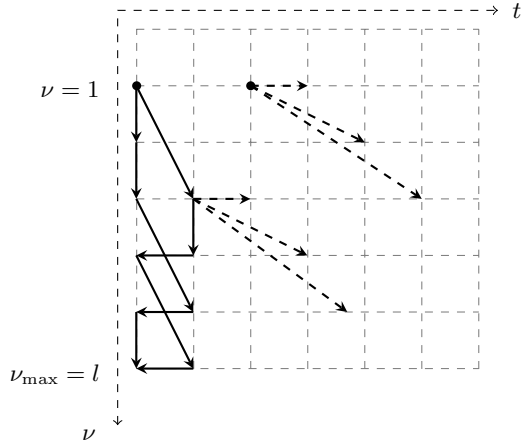
\begin{figure}[h] 
 \centering

\begin{tikzpicture}[scale=2.5]

\draw[step=.3cm, gray, thin, dashed]
(-0.902, -0.902) grid (0.9,0.9);
\draw[black, thin, dashed, <->] (-1,-1.2) -- (-1, 1) --(1,1);

\node[] at (1.1,1) {\footnotesize $t$};
\node[] at (-1.15,-1.25) {\footnotesize $\nu$};

\node[] at (-1.33, -0.93) {\footnotesize $\nu_{\max}=l$};

\node at (-0.90, 0.6) [circle,fill,inner sep=1.2pt]{};
\node[] at (-1.25, 0.58) {\footnotesize $\nu=1$};

\node at (-0.30, 0.6) [circle,fill,inner sep=1.2pt]{};

\draw[black, thick, dashed, ->, >=stealth] (-0.30, 0.6)-- (0,0.6); 
\draw[black, thick, dashed, ->, >=stealth]   (-0.30, 0.6)-- (0.6,0);  
\draw[black, thick, dashed, ->, >=stealth]  (-0.30, 0.6) --  (0.3,0.3); 

\draw[black, thick, dashed, ->, >=stealth] (-0.602,0)-- (-0.302,0); 
\draw[black, thick, dashed, ->, >=stealth]   (-0.602,0)-- (0.208,-0.6);  
\draw[black, thick, dashed, ->, >=stealth]  (-0.602,0) --  (-0.002,-0.3); 
 
\draw[black, thick, ->, >=stealth] (-0.902, 0.6) -- (-0.602,0); 
\draw[black, thick, ->, >=stealth] (-0.602, 0) -- (-0.602,-0.3); 
\draw[black, thick, ->, >=stealth]  (-0.602,-0.3) -- (-0.902, -0.3); 
\draw[black, thick, ->, >=stealth]  (-0.902, -0.3) -- (-0.602,-0.9); 
\draw[black, thick, ->, >=stealth]   (-0.602,-0.9) -- (-0.902,-0.9); 

\draw[black, thick, ->, >=stealth] (-0.902, 0.6) -- (-0.902,0.3); 
\draw[black, thick, ->, >=stealth] (-0.902, 0.3) -- (-0.902,0); 
\draw[black, thick, ->, >=stealth]  (-0.902,0) -- (-0.602, -0.6); 
\draw[black, thick, ->, >=stealth]  (-0.602, -0.6) -- (-0.902,-0.6); 
\draw[black, thick, ->, >=stealth]   (-0.902,-0.6) -- (-0.902,-0.9);

\end{tikzpicture}
\caption{The two dots represent the two summands in \eqref{eq:start_step}.  Each arrow
represents a splitting. The two solid black paths only ever  pass through the first summand in each bracket in \eqref{eq:long_recursion}. The dashed arrows  represent some splittings given by the second summands in each  bracket. }

\end{figure}

It follows immediately from \eqref{eq:long_recursion} that any  $J^{\nu}_r$ appearing in the final linear combination will have $r< \nu$ and $r\equiv \nu +1 \mod 2$. It is also easy to see that the maximal $\nu$ appearing is $\nu_{\max}=l$. We thus write
\begin{equation}\label{eq:bottom_line} \langle \partial_\phi^l|_{0}  \mathbbm 1_{K(\phi)\geq 0}, h\rangle= \sum_{r'=1}^{ \lfloor (l+1)/2 \rfloor} a_{l+1-2r'} J^l_{l+1-2r'}(u,N; h) +  \sum_{ r<\nu<l} a_{\nu,r}J^{\nu}_r(u,N; h) \end{equation}
This completes the first and main technical step in the proof.

We will now prove the proposition by induction on $k$. The statement for $k=0,1$ follows from Lemmas \ref{lem:a0_coeff} and \ref{lem:B_integral}. Suppose that it has been proved for $k'<k$. 
In the following, $h$ is either  $g_0$ or $g_0'$, and $l=k-1$.
By the induction assumption, it holds for all $k'<k$ and all $q'$ that
\[ \|I_{k'}^{n,m}(\bullet,N;f)\|_{C^{q'}(S^{m-1})}\leq c' \sum_{j=0}^{k'} \left\| \frac{\partial ^j }{\partial N^j}Cf \right\|_{C^{L'}(S^{m-1})} \] 
for $c'=c(n,m,k',q')$ and $L'=L(n,m,k',q')$. It follows by \eqref{eq:particular_h} and from the continuity of $R_D^{r}$, $C_D^r$ in the appropriate topologies, that $J^\nu_r(u, N; h)$ for $h=g_0$ and $\nu\leq k-1$, as well as for $h=g_0'$ and  $\nu\leq k-2$, satisfy
\begin{equation}\label{eq:interm_J_bound} \|J^\nu_r(\bullet, N; h)\|_{C^{q'}(S^{m-1})}\leq \tilde c \sum_{j=0}^{k'} \left\| \frac{\partial ^j }{\partial N^j}\mathcal Cf \right\|_{C^{\tilde L}(S^{m-1})} \end{equation}
for certain $\tilde c=\tilde c(n,m,k',q')$ and $\tilde L=\tilde L(n,m,k',q')$. This applies to all summands in the second sum of \eqref{eq:bottom_line}  as we assume $l=k-1$.

To understand the first sum in \eqref{eq:bottom_line}, we can ignore the second summand in each bracket in \eqref{eq:long_recursion}, as well as the second summand in \eqref{eq:start_step}, since any endpoint $J^\nu_r$ of a path of length $l-1$ passing through either of the second summands, that ends with $t=0$, must have $\nu\leq l-2$, which lands it in the second sum of \eqref{eq:bottom_line}.

We thus consider the principal term recursion
\[\partial_\phi J^{\nu}_{r,s,t}\equiv tJ_{r,s,t-1}^{\nu} -sJ_{r,s-2,t+1}^{\nu+2}-J_{r+1,s-1,t}^{\nu+1} \]
Differentiate the first summand in eq.~\eqref{eq:start_step} $l-1$ times using this recursion.  We look at paths that pass $\alpha_j$ times in the $j$th summand of \eqref{eq:long_recursion}, $j=1,2,3$, with $\alpha_1+\alpha_2+\alpha_3=l-1$. We start with $t=0$, and must have $t=0$ in the last step, so that $\alpha_1=\alpha_2$, in particular $\alpha_3\equiv l-1\mod 2$.
The path contributes towards the coefficient $a_r$ of $J^{l}_r$ with $r=\alpha_3$, and the sign of each contribution is $(-1)^{\alpha_2+\alpha_3}= (-1)^{\frac{k-2-r}{2}+r}=(-1)^{\frac{k-2+r}{2}}$. Since  for any $r\leq k-2$ of parity $k \mod 2$  there are non-zero contributions, $\sign a_r=(-1)^{\frac{k-2+r}{2}}$.

Now recall from the proof of Lemma~\ref{lem:B_integral} that by symmetry, \[\mathcal Cf(u,\phi,N)=\langle g_\phi, \mathbbm 1_{K(\phi)\geq 0}\rangle-\langle g_\phi, \mathbbm 1_{K(\phi)\leq 0}\rangle= 2 \langle g_\phi, \mathbbm 1_{K(\phi)\geq 0}\rangle.\] 
Note that $\partial_\phi ^i\mathbbm 1_{K(\phi)\geq 0}$ is by \eqref{eq:der_phi_pos} supported on $\{K(\phi)=0\}$, while $g_\phi$ vanishes on $\{K(\phi)=0\}$. Note also that $\partial_\phi \mathbbm 1_{K(\phi)\geq 0}$ is a Borel measure.
Hence
\[ \frac12\partial_\phi \mathcal  Cf(u,\phi,N)= \langle \partial_\phi g_\phi, \mathbbm 1_{K(\phi)\geq 0}\rangle \]
 and 

\begin{align}\label{eq:binomial_two_types}\frac12\partial_\phi^p\mathcal Cf(u,\phi,N)&=\sum_{i=1}^{p}\binom{p-1}{ i-1}\langle \partial_\phi^{i}|_{0} g_\phi, \partial_\phi ^{p-i}|_{0}\mathbbm 1_{K(\phi)\geq 0}\rangle \nonumber
\\&=\sum_{i=1}^{ \lfloor p/2\rfloor}\binom{p-1}{2i-1}(-1)^i\langle g_0,\partial ^{p-2i}_\phi|_{0}\mathbbm1_{K(\phi)\geq 0}\rangle \\&+ \sum_{i=0}^{ \lfloor(p-1)/2\rfloor}\binom{p-1}{ 2i}(-1)^i\langle g_0',\partial ^{p-2i-1}_\phi|_{0}\mathbbm1_{K(\phi)\geq 0}\rangle\nonumber \end{align}
Rewrite eq.~\eqref{eq:binomial_two_types} with $p=k$ as 
 \begin{align*}\label{eq:binomial_two_types}\langle g_0',\partial ^{k-1}_\phi|_{0}\mathbbm1_{K(\phi)\geq 0}\rangle =\frac12\partial_\phi^ k \mathcal Cf(u,\phi,N) &-\sum_{i=1}^{k/2}\binom{k-1}{ 2i-1}(-1)^i\langle g_0,\partial ^{k-2i}_\phi|_{0}\mathbbm1_{K(\phi)\geq 0}\rangle \\&- \sum_{i=1}^{(k-1)/2}\binom{k-1}{ 2i}(-1)^i\langle g_0',\partial ^{k-2i-1}_\phi|_{0}\mathbbm1_{K(\phi)\geq 0}\rangle\nonumber \end{align*}
By equations \eqref{eq:bottom_line} and \eqref{eq:interm_J_bound}, we see that for some $c_1=c_1(n,m,k,q')$,
\begin{equation}\label{eq:top_order_bound}\|\langle g_0',\partial ^{k-1}_\phi|_{0}\mathbbm1_{K(\phi)\geq 0}\rangle\|_{C^{q'}(S^{m-1})}\leq c_1 \sum_{j=0}^k \left\| \frac{\partial ^j }{\partial N^j}\mathcal Cf \right\|_{C^{\tilde L}(S^{m-1})}. \end{equation}

Now  set $H(u):=I_k^{n,m}(u,N;f)$. Putting $\epsilon=k\mod 2$, we note that $H\in C^\infty(S^{m-1})$ has parity $\epsilon$. Define \[ T= \sum_{r'=1}^{k/2} a_{k-2r'} R_D^{k-2r'}:C^\infty(S^{m-1})\to C^\infty(S^{m-1}),  \] 
with the coefficients from the  first sum of \eqref{eq:bottom_line}.
We have by eq.~\eqref{eq:particular_h} \[ T(H)=\sum_{r'=1}^{k/2} a_{k-2r'} J^{k-1}_{k-2r'}(u, N;  g_0')\] 
and so by eq.~\eqref{eq:bottom_line} with $l=k-1$, together with eqs. \eqref{eq:top_order_bound}  and \eqref{eq:interm_J_bound} we have 
\[  \|T(H)\|_{C^{q'}(S^{m-1})}\leq c_2\sum_{j=0}^{k} \left\| \frac{\partial ^j }{\partial N^j}\mathcal Cf \right\|_{C^{\tilde L}(S^{m-1})}   \]
Since the coefficients $a_{k-2r'}$ have alternating signs, by Lemma \ref{lem:alternating_sum_invertible} we know that $T^{-1}:C_\epsilon^{q'}(S^{m-1})\to  C_\epsilon^q(S^{m-1})$ is continuous for $q'=q+K(m)$. Taking $L(n,m,k,q)=\tilde L(n,m, k-1,q')$ and $c(n,m,k,q)=c_2$ concludes the proof.
\end{proof}

 Observe that we can view the manifold of ordered orthonormal bases in the tangent spaces to $M$ also as the total space of a principal bundle over the sphere bundle, with the  bundle projection  $\pi_0\colon FM\to SM$, $\pi_0(u_0,\ldots, u_{n-1})=u_0$. After a choice of local section $b\colon U\to FM$, the coefficients 
		$A_{\sigma\tau}^k= A_{\sigma\tau}^k(\bullet;b)$ are well-defined functions on $U\subset SM$.

\begin{corollary}\label{cor:coefficients_determined}
	The coefficients $A^k_{\sigma\tau}$ computed in subsection \ref{subsec:coefficients_integral_expressions} are determined at $p\in M$ by the restriction to $S_pM$ of all derivatives $\frac{\partial^iF}{\partial N^i}$, $i\leq k$, in all directions $N\in \Image(h_p)$. Moreover, there are constants $L=L(n,m,k)$ and $c=c(n,m,k)$ such that  for  any choice of local section  $b\colon U\to FM$
	
 	\begin{align*} \|A^k_{\sigma\tau}(\bullet;b)& \|_{C^0(S_pM\cap U)} \\
& \leq c   \sup_{v_1,\dots,v_{2k}\in S_pM}\sum_{j=0}^k \|h_p\|_\infty^{k-j}\left \|\frac{\partial^j F}{\partial (h_p(v_1,v_2)+\dots+h_p(v_{2k-1}, v_{2k}))^j}\right\|_{C^L(S_pM)}  
        	\end{align*}

\end{corollary}
\begin{proof}
	In light of Proposition \ref{prop:fiber_simiplification} and using eq.~\eqref{eq:general_coefficient}, it only remains to express the determinants $\det h^{N'}_{\sigma\tau^c}(p)$ for $N\in S(T_pM^\perp)$ as linear combinations \[\det h^{N'}_{\sigma\tau^c}(p)=\sum c_j \langle H_j,N'\rangle ^k\]  for some $c_j\in\RR$ and $H_j\in \Image(h_p)$.  As  
	$$2^k k! x_1\cdots x_k= \sum_{\epsilon \in \{1,-1\}} \epsilon_1\cdots \epsilon_k (\epsilon_1x_1+\cdots + \epsilon_kx_k)^k,$$
	we see that the products $\langle h_{\sigma(i_1),\tau^c(1)},N'\rangle\cdots \langle h_{\sigma(i_{k}), \tau^c(k)},N'\rangle$ arising in the determinant are given by universal linear combinations of monomials of the form 
	\begin{equation}
	\label{eq:h_linearcomb}
	\langle  \epsilon_1 h_{\sigma(i_1),\tau^c(1)}+\cdots+  \epsilon_k h_{\sigma(i_{k}), \tau^c(k)}, N'\rangle^k	
	\end{equation}
with $\epsilon_j\in\{1,-1\}$. This proves the first part. 
		 Proposition~\ref{prop:fiber_simiplification} with $q=0$ readily gives the desired bound on the norm. 
	\end{proof}

\begin{proof}[Proof of Theorem \ref{thm:jet_weyl}.]
The statement is contained in the  first part of Corollary~\ref{cor:coefficients_determined}. 
\end{proof}

\begin{remark}
	The function $F\in C^\infty(S^{n-1})$ need not be the restriction of a norm; formula \eqref{eq:general_coefficient} and Corollary \ref{cor:coefficients_determined} apply to arbitrary smooth $F$. Observe also that $A^k_{\sigma\tau}$ depends linearly on $F$.
\end{remark}

\section{The weak Weyl principle} \label{sec:main_proof}

\subsection{Geometric preliminaries}
Recall the  tautological map $\theta:SM\to S(V)$, $\theta(x,v)=v$. The second fundamental form of $M$ is $h:\Sym^2(TM)\to TM^\perp$,  $h(X,Y)= (\overline \nabla_X Y)^\perp$.

\begin{lemma}\label{lem:theta_vs_h}
	$\Image(d_{p,u}\theta)=T_u\theta(S_pM)\oplus \Image(h_p(u,\bullet))$.
	\end{lemma}
\begin{proof}
	In one direction, first observe that $T_u\theta( S_pM)\subset \Image(d_{p,v}\theta)$.   If $v_0\in  T_p M$ is nonzero, $p(t)$ is a curve through $p$ with $p'(0)=v_0$ and $u(t)$ is the parallel transport of $u$ along $p(t)$, then putting $\gamma(t)=(p(t), u(t))$ we find
	\[ h_p(u,v_0)=u'(0)= d_{p,u}\theta(\gamma'(0)).   \]
	 Thus $\Image(h_p(u,\bullet))\subset \Image(d_{p,u}\theta)$. 
	In the other direction, let $\gamma(t)=(p(t), u(t))$ be arbitrary with $u(t)\in SM$  and $\gamma(0)=(p,u)$. Then $d\theta(\gamma'(0))=u'(0)=h_p(u,p'(0))+\nabla_{p'(0)}u(t)$. It remains to note that $\nabla_{p'(0)}u(t)\in T_pM$, and $\langle \nabla_{p'(0)}u(t), u\rangle=0$.
\end{proof}

\begin{corollary}\label{cor:dir_reg}
$M$ is directionally regular at $(p,u)\in SM$ if and only $h_p(u,\bullet)$ has rank $\min(m,n-m)$.
\end{corollary}
\begin{proof} Follows immediately from Lemma~\ref{lem:theta_vs_h}.
\end{proof}

 We now show that most immersions are directionally regular at most points.  Equip $C^\infty(M,V)$ with the smooth Whitney topology, see \cite{guillemin_golubitsky} for details. A \emph{residual subset} is the countable intersection of open dense subsets.
	
\begin{lemma}\label{lem:generic_perturbation}

Let $W\subset C^\infty(M,V)$ be the subset of smooth immersions of the manifold $M^m$ in the linear space $V^n$.  Then there is a  dense residual subset $W'\subset W$ such that for all $f\in W'$, $f$ is directionally regular on a dense open subset of $SM$. 
\end{lemma}
\proof
 For basic notions of jet bundles of smooth functions, we refer to \cite{guillemin_golubitsky}. By \cite{guillemin_golubitsky}*{Lemma 5.5}, $W$ is an open subset of $C^\infty(M,V)$.

For $f\in C^\infty(M,V)$, let $\tilde \theta_f: TM\to  V$ be the composition of $df: TM\to TV$ with the projection to the second factor of $TV=V\times V$, which is just the previously encountered tautological map, extended to arbitrary smooth maps.
 
Denote $k=\min\{n,2m\}$. For $(q,u)\in TM$, $u\neq 0$, define the subset $W_{q,u}:=\{j_q^2f:  \mathrm{rank}(d_{q,u}\tilde \theta_f)<k\}$ 
of the bundle of $2$-jets $J^{2}(M,V)$. Note that $W_{q,u}$ is a finite union of  locally closed submanifolds, each of codimension at least $m+1$.  By Thom's transversality theorem \cite{guillemin_golubitsky}*{Theorem 4.9}, each set $\{f\in C^\infty(M,V): j^{2}f\transv W_{q,u}\}$ is   residual in $C^\infty(M,V)$. Note that $j^2f(M)\subset J^2(M,V)$ is an $m$-dimensional submanifold, and therefore $j^{2}f\transv W_{q,u}\iff j^{2}f(M)\cap W_{q,u}=\emptyset$.
Now choose a dense sequence $(q_i,u_i)\in  TM\setminus \underline 0$. By \cite{guillemin_golubitsky}*{Lemma 3.3}, any residual subset of $C^\infty(M,V)$ is dense, and it follows that the intersection 
$$W':=W\cap \bigcap_{i\geq 1} \{f\in C^\infty(M,V): j^{2}f(M) \cap W_{q_i,u_i}=\emptyset\}$$ is a dense residual subset of $W$. 
Finally, note that the set of directionally regular points of any given immersion is open. Thus any $f\in W'$ is directionally regular at an open subset containing all $(q_i,u_i)$, concluding the proof. 
\endproof

\subsection{When WWP fails.}
 As before, in the following $M\subset V$ is an immersed submanifold, $m=\dim M$ and $n=\dim V$.
We first prove Theorem \ref{thm:false_in_general}, which is implied by the following two statements.
	
	\begin{lemma}\label{lem:hemispherical_equator_essential}
		Consider the hemispherical transform $H:C^\infty(S^{m-1})\to C^\infty(S^{m-1})$, and let $a:S^{m-1}\to S^{m-1}$ be the antipodal map. Assume $u,v\in S^{m-1}$ and $u\perp v$. Let $U$ be any neighborhood of $u$. Then there is an odd function $f\in C^\infty_c(U\cup aU)$ such that $H^{-1}(f)(v)=1$.
	\end{lemma}

	\begin{proof}
		Let $\tilde f$ be a smooth, odd,  $\mathrm{Stab}(u)=\SO(m-1)$-invariant function supported in $U\cup aU$ with $d_u\tilde f\neq 0$. Set $\tilde g=H^{-1}\tilde f$, which is clearly $\SO(m-1)$-invariant. 
		Note that the hypersphere $u^\perp$ must lie entirely in the support of $\tilde g$, otherwise by $\SO(m-1)$-invariance there is an open band around $u^\perp$ where $\tilde g$ vanishes, but then $\tilde f=H\tilde g$ must be constant in a neighborhood of $u$.
		We may then find a small rotation $R\in\SO(m)$ such that $R\tilde f$ is supported in $U\cup aU$, and $H^{-1}(R\tilde f)(v)=R\tilde g(v)\neq 0$. It remains to set $f= \frac{1}{R\tilde g(v)}R\tilde f$.
	\end{proof}
	
\begin{proposition}\label{prop:thm B}
	Assume $m=\dim M\geq 3$, and  $\theta\colon SM\to S^{n-1}$ has constant rank $r\leq n-2$ near $(p,u)\in SM$. Assume moreover that $\Image(h_p)\neq\Image(h_p(u,\bullet))$. Then there is a neighborhood $U\subset M$ of $p$ such that the coefficients $A^1_{\sigma\tau}$ at $(p,u)$ are not determined by $F|_U$.
\end{proposition}
\begin{proof}
	Since $\Image(h_p)\neq \Image(h_p(u,\bullet))$, we can find  $e_1\in S_pM$, $e_1\perp u$ with $h_p(e_1, e_1)\notin \Image(h_p(u,\bullet))$. Fix some $e_0\in S_pM$ such that $e_0\perp e_1,u$. Then $A_{\sigma\tau}^1(e_0)=\frac12 H^{-1}(\partial_\phi|_{0}F(u',\phi, h_{11}))(e_0)$ with $\sigma=\{1\},\tau=\{2,\dots,m-1\}$ by eq.~\eqref{eq:a1_coefficient}. 
	
	As $\theta$ has constant rank near $(p,u)$, the same is by Lemma~\ref{lem:theta_vs_h} true near $(p,-u)$. Hence there are open neighborhoods $U\subset M$ of $p$ and $W\subset S^{n-1}$ of $u$ such that $Z:=\theta(SU)\cap (W\cup a W)\subset S^{n-1}$ is an embedded submanifold, and $T_{u}Z=T_{u}\theta(S_pM)\oplus \Image(h_p(u,\bullet))$. As  $h_{11}=h_p(e_1,e_1)$ is transversal to $Z$ at $u$, we see that $u'\mapsto \frac{\partial}{\partial h_{11}}F(u')$, $u'\in S_pM$, can be perturbed arbitrarily in a small neighborhood of $\{u, -u\}\subset S_pM$ by perturbing $F$, while keeping $F|_{\theta(SU)}$ fixed and $F$ symmetric. Such a perturbation, if small enough, will not change the convexity of $F$, but by Lemma \ref{lem:hemispherical_equator_essential}, $H^{-1}(u'\mapsto \frac{\partial}{\partial h_{11}}F(u'))(e_0)$ will not be determined by $F|_{\theta(SU)}$.

\end{proof}

\begin{proof}[Proof of Theorem~ \ref{thm:false_in_general}]

	By directional regularity, $\theta$ has constant rank in a neighborhood of $(p,u)$.
	 Clearly, 
	$\dim \Image(h_p(u,\bullet))\leq m$.  Hence if $\dim \mathcal O_p^2 M>2m$, then $\Image(h_p)\neq \Image(h_p(u,\bullet))$. The statement follows now from Proposition~\ref{prop:thm B}
\end{proof}

When $m+3\leq n\leq 2m$, WWP generically holds by Theorem \ref{thm:special_dimension}, but can fail for particular submanifolds, as the following examples show.
\begin{example}\label{exm:M3R6}

Consider a generic surface $\Sigma^2\subset \RR^5$, which has $\dim\Image h_q^\Sigma=3$ for all $q\in\Sigma$, and $\dim \Image h_q^\Sigma(v,\bullet)=2$ for $(q,v)$ in an open dense subset of $ S\Sigma$. Now take $M=\Sigma\times \RR\subset \RR^6$. Writing $u=(v,\xi)\in T_{q,s}M=T_pM\oplus \RR$, we have $h_{q,s}(u_1,u_2)=h^\Sigma_q(v_1,v_2)$. Hence $\dim \Image h_p=3$ everywhere, while $\dim \Image h_p(u,\bullet)=2$  for $(p,u)$ in an open and dense subset of $ SM$. Fixing such $(p,u)$, $d\theta$ has constant rank $4$ in its neighborhood, and Proposition \ref{prop:thm B} shows that WWP fails.  Note that by taking $M= \Sigma\times \RR^{k}$ we can produce higher dimensional examples of codimension $3$ for which WWP fails.
\end{example}
\begin{example}

For a less trivial example, let $M^4\subset \RR^8$ be given by \[M=\{(x_1,x_2,x_3,x_4,x_1x_3,x_1x_4,x_2x_3,x_2x_4)\},\] which is nothing but the Segr\'e embedding $\mathbb P^2\times \mathbb P^2\hookrightarrow\mathbb P^8$ in affine charts. 
The second fundamental form can be represented in a certain frame of the normal bundle by \[h_x(u, v)=(u_1v_3+u_3v_1,u_1v_4+u_4v_1, u_2v_3+u_3v_2,u_2v_4+u_4v_2).\] Hence $\dim \Image h_x=4$ everywhere, while $\dim \Image h_x(u,\bullet)=3$ for generic $u$, and so Proposition \ref{prop:thm B} can be applied to deduce that WWP fails for $M$. By taking products with $\RR^k$  we can promote $M$ to higher dimensional examples  of codimension $4$ for which WWP fails.
\end{example}

\subsection{When WWP holds.}
We now turn to the positive results. The weak Weyl principle ultimately follows from the relationship between the derivatives of $F$ along the tangent and normal directions to the manifold.

\begin{proposition}\label{prop:weyl_equivalent_to_h}
	Assume $M^m\subset V^n$, and  $\theta:SM\to S^{n-1}$ is submersive in a dense subset of $(p,v)\in SM$. Then the coefficients $A^k_{\sigma\tau}(p,u;F)$ on $S_pM$ are determined by $(M, F|_M)$ for all $k\geq 0$. 
	
\end{proposition}
\begin{proof}
 By assumption, we have $\Image(h_p)=\Image(h_p(v,\bullet))=T_pM^\perp$ for $(p,v)$ in a dense subset $D\subset SM$.
It suffices by Corollary \ref{cor:coefficients_determined} to show that for any $u\in S_pM$, and any $N\in T_pM^\perp$, $\frac{\partial ^k F}{\partial N^k}(u)$ is determined by $F|_{\theta(SM)}$. 

Assume first that $(p,u)\in D$. Then $\theta(SM)\subset S^{n-1}$ contains a neighborhood of $u$ in $S^{n-1}$, concluding this case.
For general $(p,u)\in SM$ and $N\in T_pM^\perp$, we choose sequences $D\ni(p_j,u_j)\to (p,u)$, $T_{p_j}M^\perp\ni N_j\to N$, and obtain $\frac{\partial ^kF}{\partial N^k}(u)=\lim_{j\to \infty}\frac{\partial ^k F}{\partial N_j^k}(u_j)$.
\end{proof}

\begin{lemma}\label{lem:first_coefficients_determined}
	Assume $M^m\subset V^n$,   and $\Image(h_p)=\Image(h_p(u,\bullet))$ in a dense subset of $(p,u)\in SM$. Then the coefficients $A^1_{\sigma\tau}(p,u;F)$ on $S_pM$ are determined by $(M, F|_M)$. 
\end{lemma}
\begin{proof}
	Consider first $(p,u)$ in the given dense set $D\subset SM$. Then for any $N\in \Image (h_p)$ we can find $v\in T_pM$ such that $h_p(u,v)=N$.  Choose a curve $\gamma(t)$ with $\gamma(0)=p$, $\gamma'(0)=v$, and let $u(t)$ be the parallel transport of $u$ along $\gamma(t)$. It follows that 
	\[ \frac{\partial F}{\partial N}(u)=\left.\frac{d}{dt}\right|_{t=0}F(u(t)).\]
	is determined by $F|_M$. We conclude the proof by continuity as in Proposition \ref{prop:weyl_equivalent_to_h}.
\end{proof}

\begin{lemma}\label{lem:codimension2}
	Assume $M^m\subset V^n$ and $n=m+2$. If $\dim \Image h_y=1$ in some open neighborhood of $p\in M$, then either 	
	a neighborhood of $p$ in $M$ lies in a hyperplane, or $h_p:T_pM\times T_pM\to \Image(h_p)$ is a symmetric form of rank one.
\end{lemma}
\begin{proof}
	We may assume $p=0$, and $M$ is the graph of $F=(f,g):\RR^m\to \RR^2$, with $\nabla f(0)=\nabla g(0)=0$.
	Denote $\overline F(x)=(x,F(x)):\RR^m\to \RR^{m+2}$, and for $v\in\ \RR^m$ write
	$$\overline v= \overline F_* v|_x= (v, \pder[F]{v}(x))\in T_{\overline F(x)}M.$$
	Observe that the vector fields
	$$ N_f= (\nabla f, -1,0), \qquad N_g= (\nabla g,0,-1)$$
	span $T_{\overline F(x)}M^\perp$ for every $x\in \RR^m$.
	
	If $y=\overline F(x)$ and $u,v\in \RR^m$, then 
		$$\langle h_y(\overline u, \overline v), N_f\rangle = \frac{\partial^2 f}{\partial u\partial v}(x)$$
	and 
	$$\langle h_y(\overline u, \overline v), N_g\rangle = \frac{\partial^2 g}{\partial u\partial v}(x).$$

	Writing $h_y(u, v)=a(u,v)N_f+b(u,v)N_g$, we find \[a(u,v)=D^{-1}(|N_g|^2 \frac{\partial^2 f}{\partial u\partial v}-\langle N_f, N_g\rangle  \frac{\partial^2 g}{\partial u\partial v}) \]
\[b(u,v)=D^{-1}(-\langle N_f, N_g\rangle  \frac{\partial^2 f}{\partial u\partial v}+|N_f|^2 \frac{\partial^2 g}{\partial u\partial v}), \]
where $D=|N_f|^2|N_g|^2-\langle N_f, N_g\rangle ^2$.

Thus $\dim \Image h_y=1$ implies 
there exist smooth functions $\tilde \alpha,\tilde \beta\colon U\to \RR$ on a neighborhood of $0\in \RR^m$, $(\tilde\alpha,\tilde\beta)\neq (0,0)$, such that for all $u,v\in \RR^m$,
$$\tilde \alpha a(u,v)+\tilde \beta b(u,v) =0.$$

	Plugging in the expressions for $a,b$ we find that there exist smooth functions $\alpha, \beta\colon U\to \RR$, $(\alpha,\beta)\neq (0,0)$ such that $$ \alpha(x) \nabla^2 f(x) + \beta(x) \nabla^2 g(x)=0.$$
			We may assume $\beta(0)\neq 0$, and so setting $\gamma=-\alpha/\beta$ we find \begin{equation}\label{eq:hessians_proportional}\nabla^2g=\gamma \nabla^2f.\end{equation}

	If $\gamma$ is constant near $0$, it follows that $g=\gamma f$ near $0$, so $M$ locally lies in a hyperplane. 
	Otherwise, $0\in \supp(\nabla \gamma)$, and we choose a sequence $\RR^m\ni z_k\to 0$ with $\nabla \gamma(z_k)\neq 0$. Given $k$, we can find Euclidean coordinates $x_i$ on $\RR^m$ such that $\frac{\partial^2 f}{\partial x_i\partial x_j}(z_k)=0$ for all $i\neq j$. By assumption, there is $i=i(k)$ such that $\frac{\partial \gamma}{\partial x_i}(z_k)\neq 0$.  From \eqref{eq:hessians_proportional} we find $$\frac{\partial\gamma}{\partial x_i}\frac{\partial ^2f}{\partial x_j^2}=\frac{\partial\gamma}{\partial x_j}\frac{\partial ^2f}{\partial x_i\partial x_j}, $$ so that for all $j\neq i$ we have $\frac{\partial ^2f}{\partial x_j^2}(z_k)=0$. Thefefore $\rank \nabla f^2(z_k) \leq 1$ for all $k$, and consequently $\rank \nabla f^2(0) \leq 1$. This readily implies that $h_p$ has rank at most one, and since $h_p\neq 0$ we conclude $\rank h_p=1$.
	
	\end{proof}
 We will also need a fact from linear algebra.
 
 \begin{lemma}\label{lem:linear_algebra}
 	Let $h\colon \Sym^2(V)\to W$ be a linear map. If $\dim \Image h\leq  2\leq \dim V$, then there exists an open and dense subset $U\subset V$, such that for every  $u\in U$, the map $V\to W$, $v\mapsto h(u,v)$ is onto $\Image h$.
 \end{lemma}
 \begin{proof}
 	
 	Let us first assume that $\dim \Image h=1$. The set of all $u\in V$ such that $h(u,v)=0$ for all $v\in V$ is a linear subspace of $V$. Since $h\neq 0$, this subspace must be different from $V$. Its complement is therefore open and dense. 
 	
 	Let us assume now  that $\dim \Image h=2$. Note that without loss of generality we may assume that $h$ is surjective, that is $\dim W=2$. 
 	Choose a basis $w_1, w_2$ of $W$. With respect to this basis write $h=(h^1,h^2)$. For every 
 	$u\in V$ we have $h^1(u,\bullet),h^2(u,\bullet)\in V^*$. Let $Z$ be the set of all $u\in V$ such that 
 	the functionals $h^1(u,\bullet), h^2(u,\bullet)\in V^*$ are linearly dependent. Clearly $Z$ is the zero set of a system of polynomial equations on $V$.  Let $U$ be its complement.  Then either $U=\emptyset$, or $U$ is open and dense.  In the latter case we are done. 
 	
 	Otherwise, $h^1(u,\bullet), h^2(u,\bullet)\in V^*$ are linearly dependent for every $u\in V$. This implies that for every $u\in V$ there exists a non-zero $\lambda(u)\in W^*$ such that $h(u,v)\in \ker \lambda(u) $ for every $v\in V$. 
 	
 	 	Let $v_1,v_2\in V$. If
 	$h(v_1,v_2)=0$, then since $\{h(v_1+v_2,u)\colon u\in V\}$ is contained in the line $\Ker(\lambda(v_1+v_2))$, we have in particular that $h(v_1+v_2,v_1)=h(v_1,v_1)$, $h(v_1+v_2, v_2)=h(v_2,v_2)$ are proportional.
 	If $h(v_1,v_2)\neq 0$, then since $h(v_1,v_2)=h(v_2,v_1)$, we must have $\ker \lambda(v_1)=\ker \lambda(v_2)$, hence  $h(v_1,v_1)$ and $h(v_2,v_2)$ must again be proportional. We conclude that $\dim \Image  h=1$, in contradiction.

 \end{proof}
 \textit{Proof of Theorem \ref{thm:special_dimension}}.
 Recall that  by Proposition \ref{prop:k=1_is_enough}, we only need to consider $\mu_1^F$. Let us show that all $A^k_{\sigma\tau}(p,u)$ are determined by the restriction of $F$ to a neighborhood of $p$ in $M$.
 
 Assume first $n=m+1$, and consider $p\in M$. If $h_p=0$ then  $\partial^k F/\partial h_p(v_1,v_2)^k=0$ for all $k\geq 1$ and $v_1,v_2\in T_pM$, and we conclude using Corollary \ref{cor:coefficients_determined}.
 If $h_p\neq 0$, it holds by Lemma \ref{lem:linear_algebra} that $\Image h_p(u,\bullet)=\Image h_p=T_pM^\perp$ in a dense set of $(p,u)\in SM$ near $p$. Proposition \ref{prop:weyl_equivalent_to_h} now concludes this case.
 
 Assume now $n=m+2$, and fix $p\in M$. If $h_p=0$, we conclude as in the previous case.
 If $h_p\neq0$, then by the lower semicontinuity of $\dim \Image h_x$ either  $\dim \Image h_x=1$ in a neighborhood of $p$, or $p$ is a limit point of $\{x\in M: \Image h_x=T_xM^\perp\}$.   Hence by Lemma \ref{lem:codimension2}, one of the following three must hold:
 
 \begin{enumerate}
 	\item $p$ is a limit point of $W:=\{x\in M: \text{a neighborhood of }x\text{ lies in a hyperplane} \}$.
 	\item For every $x\in M$ in a neighborhood of $p$, it holds that $\dim \Image(h_x)=1$, and for any $N\in T_xM^\perp$ we have $\rank\langle h_x, N\rangle\leq 1$.
 	\item $p$ is a limit point of $\{x\in M: \Image h_x=T_xM^\perp\}$.
 \end{enumerate} In the first case, we use the previously considered hypersurface case and the linear Weyl principle to conclude that for some neighborhood $U$ of $p$, $A^k_{\sigma\tau}$ is determined in $U\cap W$ by the restriction of $F$ to $U$, and by continuity it is also determined at $p$. In the second case it follows from  \eqref{eq:general_coefficient} that $A^k_{\sigma\tau}(u)=0$ for $u\in S_pM$ and $k\geq 2$. Recall also that $A^0(u)$ is intrinsically determined by Corollary \ref{cor:a0_intrinsic}. By Lemma \ref{lem:linear_algebra}, it holds in a dense set of $(x,u)$ in $SM$ near $p$ that $\Image h_x=\Image h_x(u,\bullet)$; hence by Lemma \ref{lem:first_coefficients_determined} all coefficients $A^1_{\sigma\tau}(u)$ are intrinsically defined, concluding the proof in this case. In the last case, we conclude with Lemma \ref{lem:linear_algebra} and Proposition \ref{prop:weyl_equivalent_to_h}.

 Finally for $m+3\leq n\leq 2m$,  the statement follows from Proposition \ref{prop:weyl_equivalent_to_h} and Lemma \ref{lem:generic_perturbation}.
\qed

\section{Concluding remarks and open problems} \label{sec:prolbems}

One remaining open question is whether the full Weyl principle holds in codimensions one and two, namely whether the restriction of the Holmes-Thompson intrinsic volumes is independent of the isometric immersion $M\looparrowright V$. While at first glance it appears unlikely due to the appearance of the second fundamental form in \eqref{eq:general_coefficient}, rigidity phenomena might have the final say in this matter. 

For instance, Weyl's theorem for generic Riemannian manifolds of dimension at least three that are embedded as hypersurfaces in $\RR^n$ is trivial, since in this case the first fundamental form determines the second fundamental form up to sign, see \cite{spivak5}. Assuming a similar phenomenon occurs in the Finsler setting, the validity of the full Weyl principle appears plausible in light of the weak Weyl principle for hypersurfaces. In codimension two no such rigidity takes place in the Riemannian setting, and consequently we conjecture that the full Weyl principle fails in the Finslerian setting in codimension two.

 In addition, it remains unknown for $m+3\leq n\leq 2m$ whether the set of immersions $M^m\looparrowright V^n$ for which WWP holds contains in fact a dense open set, rather than merely a residual set, as established in Theorem \ref{thm:special_dimension}.

\end{document}